%% file: pcv-ejp-rev3.tex
\documentclass[12pt]{amsart}
\usepackage{geometry}
\usepackage{graphicx}
\usepackage{amssymb}
\usepackage{epstopdf}
\usepackage{amsmath,amscd}
\usepackage{amsthm}
\usepackage{enumerate}
\usepackage{url,verbatim,soul}
\usepackage{subfig}

\RequirePackage[colorlinks,citecolor=blue,urlcolor=blue]{hyperref}
\usepackage{breakurl}
\theoremstyle{plain}
\newtheorem{theorem}{Theorem}

\newtheorem{lemma}[theorem]{Lemma}
\newtheorem{proposition}[theorem]{Proposition}

\newtheorem{example}[theorem]{Example}

\newtheorem{remark}[theorem]{Remark}

\newcommand\es{\varnothing}
\newcommand\wt{\widetilde}
\newcommand\ol{\overline}

\newcommand\EE{{\mathbb E}}

\newcommand\sF{{\mathcal F}}
\newcommand\RR{{\mathbb R}}
\newcommand\ZZ{{\mathbb Z}}

\newcommand\QQ{{\mathbb Q}}
\newcommand\PP{{\mathbb P}}

\renewcommand\ul{\underline}

\renewcommand\a{\alpha}
\newcommand\om{\omega}
\newcommand\g{\gamma}

\newcommand\si{\sigma}
\newcommand\eps{\epsilon}
\newcommand\De{\Delta}
\newcommand\qq{\qquad}
\newcommand\q{\quad}
\newcommand\resp{respectively}

\newcommand\oo{\infty}
\newcommand\sG{{\mathcal G}}

\newcommand\sP{{\mathcal P}}

\newcommand\Om{\Omega}
\newcommand\La{\Lambda}
\newcommand\Si{\Sigma}
\newcommand\de{\delta}
\newcommand\la{\lambda}

\newcommand\rad{\text{\rm rad}}
\newcommand\bigmid{\,\big|\,}
\newcommand\biggmid{\,\bigg|\,}

\newcommand\bY{\mathbf{Y}}
\newcommand\bz{\mathbf{z}}
\newcommand\ws{\textsc{ws}}

\renewcommand\ell{l}

\newcommand\pd{\partial}

\newcommand\sm{\setminus}

\renewcommand\o{{\mathrm o}}

\newcommand\Po{\text{\rm P}}
\newcommand\Eo{\text{\rm E}}
\newcommand\Ss{\text{\rm S}}
\newcommand\Ii{\text{\rm I}}
\newcommand\Rr{\text{\rm R}}
\newcommand\Cps{\mbox{\rm C$_{\ga,\si}^{\hskip1pt\prime}$}}
\newcommand\Cs{\mbox{\rm C$_{\ga,\si}$}}
\newcommand\oladd{\ul\alpha_{\text{\rm dd}}}

\newcounter{mycount}

\newenvironment{numlist}{\begin{list}{\arabic{mycount}.}%
   {\usecounter{mycount}\labelwidth=1cm\itemsep 0pt}}{\end{list}}
\newenvironment{letlist}{\begin{list}{\rm(\alph{mycount})}%
   {\usecounter{mycount}\labelwidth=1cm\itemsep 0pt}}{\end{list}}
\newenvironment{Alist}{\begin{list}{\MakeUppercase{\alph{mycount}}.}%
   {\usecounter{mycount}\labelwidth=1cm\itemsep 0pt}}{\end{list}}

\newcommand\bX{{\mathbf X}}
\newcommand\bx{{\mathbf x}}
\newcommand\bu{{\mathbf u}}
\newcommand\by{{\mathbf y}}
\newcommand\bv{{\mathbf v}}
\newcommand\bw{{\mathbf w}}

\newcommand\ac{\alpha_{\text{\rm c}}}
\newcommand\covid{{\textsc{Covid}}-19}
\newcommand\s{\sigma}
\newcommand\Pda{\PP_{\la,\alpha}}
\newcommand\Eda{\EE_{\la,\alpha}}

\newcommand\lac{\la_{\text{\rm c}}}

\newcommand\ga{\gamma}
\newcommand\thd{\theta_{\text{\rm d}}}
\newcommand\thdd{\theta_{\text{\rm dd}}}

\newcommand\iNT{\text{\rm Int}}
\newcommand\ula{\underline{\lambda}}
\newcommand\oac{\underline{\alpha}}
\newcommand\eqd{\stackrel{\text{\rm d}}{=}}
\newcommand\bde{\boldsymbol{\delta}}
\newcommand\bzeta{\boldsymbol{\zeta}}
\newcommand\ua{\underline\alpha}
\newcommand\vtau{\vec\tau}

\numberwithin{equation}{section}
\numberwithin{theorem}{section}
\numberwithin{figure}{section}

\title[Brownian snails with removal]
{Brownian snails with removal:\\  
epidemics in diffusing populations}

\author{Geoffrey R.\ Grimmett}
\address{(GRG) Statistical Laboratory, Centre for
Mathematical Sciences, Cambridge University, Wilberforce Road,
Cambridge CB3 0WB, UK} 
\address{School of Mathematics \&\ Statistics, The University of Melbourne, Australia}
\email{g.r.grimmett@statslab.cam.ac.uk}
\urladdr{\url{http://www.statslab.cam.ac.uk/~grg/}}
\address{(ZL) Department of Mathematics,
University of Connecticut, Storrs, Connecticut 06269-3009, USA} \email{zhongyang.li@uconn.edu}
\urladdr{\url{http://www.math.uconn.edu/~zhongyang/}}

\author{Zhongyang Li}
\date{20 October 2020, revised 6 April 2022} %\today}


\begin{document}
\begin{abstract}
Two stochastic models of susceptible/infected/removed (SIR)  type
are introduced for the spread of infection through a spatially-distributed population. 
Individuals are initially distributed at random in space, 
and they move continuously according to independent diffusion processes. The disease may
pass from an infected individual to an uninfected individual when they are sufficiently close.
Infected individuals are permanently removed at some given rate $\alpha$. 
Such processes are reminiscent of so-called frog models, 
but differ through the action of removal, as well as the fact that frogs jump whereas snails slither.

Two models are studied here, termed the \lq delayed diffusion' and 
the \lq diffusion' models. In the first,
individuals are stationary until they are infected, 
at which time they begin to move; in the second, all individuals
start to move at the initial time $0$. 
Using a perturbative argument, conditions are established under 
which the disease infects a.s.\ only finitely many individuals.
It is proved for the delayed diffusion
model that there exists a critical value $\ac\in(0,\oo)$ for the survival of the epidemic.  
\end{abstract}

\maketitle

\section{Introduction}\label{sec:int}

\subsection{Outline of the models}\label{ssec1-1}
Numerous mathematical models have been 
introduced to describe the  spread of a disease around a population.
Such models may be deterministic or stochastic, or a mixture of each; 
they may incorporate a range of
factors including susceptibility, infectivity, recovery, and removal; the population members 
(termed \lq particles') may be distributed about
some given space;  and so on. We propose two models
in which (i) the particles move randomly about the space that they inhabit, (ii)
infection may be passed between
particles that are sufficiently close to one another, and (iii) after the elapse of a random time since infection,
a particle is removed from the process. These models differ from that of Beckman, Dinan, Durrett,
Huo, and Junge  \cite{BDDHJ} 
through the introduction of the permanent \lq removal'
of particles, and this new feature brings a significant new difficulty to the 
analysis.

We shall concentrate mostly on the case in which the particles
inhabit  $\RR^d$ where $d \ge 2$.
Here is a concrete example of the processes studied here.
\begin{letlist}
\item Particles are initially distributed in $\RR^d$ in the manner of a  rate-$\la$ Poisson process
conditioned to contain a point at the origin $0$.
\item Particles move randomly within $\RR^d$ according to independent Brownian motions with 
variance-parameter $\s^2$. 
\item At time $0$ the particle at the origin (the initial \lq infective') suffers from an infectious disease, 
which may be passed to others when sufficiently close.
\item When two particles, labelled $P$ and $P'$, are within a given distance $\delta$, and $P$ is already  infected,
then particle $P'$ becomes infected. 
\item Each particle is infected for a total
period of time having the exponential distribution with parameter $\alpha\in[0,\oo)$, and is then permanently removed. 
\end{letlist}
The fundamental question is to determine for which vectors $(\la,\delta,\s,\alpha)$  it is the case that 
(with strictly positive probability) infinitely many particles become infected. 
For simplicity, we shall assume henceforth that 
\begin{equation}\label{g1}
\de=\s=1.
\end{equation}
We shall generally assume $\alpha>0$. In the special case $\alpha=0$, (studied in \cite{BDDHJ})
a particle once infected remains infected forever, and the subsequent
analysis is greatly facilitated by a property of monotonicity that is absent in the more challenging 
case $\alpha>0$ considered in the current work.

Two protocols for movement feature in this article.
\begin{enumerate}
\item[A.] \emph{Delayed diffusion model.} The initial infective starts to move at time $0$, 
and all other particles remain stationary until they are infected, at which times they
begin to move.
\item[B.] \emph{Diffusion model.} All particles begin to move at time $0$.
\end{enumerate}

The main difficulty in studying these models arises from the fact that particles are permanently
removed after a (random) period of infectivity. This introduces a potential non-monotonicity into the model,
namely that the presence of infected particles may hinder the growth of the process through the
creation of islands of \lq removed' particles that can act as barriers
to the further spread of infection. A related situation (but without 
the movement of particles) 
was considered by Kuulasma \cite{KK} in a discrete setting, and the methods 
derived there are useful in our Section \ref{sec:3} (see also Alves et al.\ \cite[p.\ 4]{AMP02a}).
This issue may be overcome for the delayed diffusion model, but remains
problematic in the case of the diffusion process. 

Let $I$ denote the set of particles that are ever infected, and 
\begin{equation}\label{g3}
\theta(\la,\alpha) := \Pda(|I|=\oo).
\end{equation}
We say the process 
\begin{align*}
\text{becomes extinct }\q  &\text{if }\theta(\la,\alpha)=0,\\
\text{survives }\q &\text{if } \theta(\la,\alpha)>0.
\end{align*}
Let $\lac$ denote the critical value of $\la$ for the disk (or \lq Boolean') percolation model with radius $1$ on $\RR^d$
(see, for example,  \cite{MR}). It is immediate for both models above that 
\begin{equation}\label{eq:perc}
\theta(\la,\alpha)>0 \qq \text{if  $\la>\lac$ and $\alpha\ge 0$},
\end{equation}
since in that case the disease spreads instantaneously to the percolation cluster $C$ containing the initial infective, 
and in addition we have $\Pda(|C|=\oo)>0$. 

\subsection{Two exemplars of results}\label{ssec1-2}
We write $\thd$ (\resp, $\thdd$) for the function $\theta$ of \eqref{g3}
in the case of the diffusion model
(\resp, delayed diffusion model). The following two theorems are proved
in Sections \ref{sec:dd} and \ref{sec:d} as special cases of results for more general epidemic
models than those given above.  

\begin{theorem}[Brownian delayed diffusion model]\label{thm:1}
Let $d \ge 2$. There exists a non-decreasing function 
$\ac:(0,\oo)\to (0,\oo]$ such that 
\begin{equation}\label{g4}
\thdd(\la,\alpha) \begin{cases} 
=0 &\text{if } \alpha>\ac(\la),\\
>0 &\text{if } \alpha<\ac(\la).
\end{cases}
\end{equation} 
Furthermore, $\ac(\la)=\oo$ when $\la>\lac$, and 
there exists $\ula\in(0,\lac]$ such that $\ac(\la)<\oo$ when $0<\la<\ula$.
\end{theorem}

The delayed diffusion model has no phase transition when $d=1$; see Theorems \ref{thm:nopt} and \ref{thm:nopt+}.

\begin{theorem}[Brownian diffusion  model]\label{thm:2}
Let $d \ge 1$.
There exists $\ula\in(0,\lac]$ and a non-decreasing function 
$\ua:(0,\ula)\to(0,\oo)$ such that 
$\thd(\la,\alpha)=0$ when $\alpha>\ua(\la)$
and $0<\la<\ula$.
\end{theorem}

For the diffusion model, we have no proof of survival for $d\ge 2$ and small positive $\alpha$
(that is,  that $\thd(\la,\alpha)>0$ for some $\la<\lac$ and  $\alpha>0$),
and neither does the current work answer the question of whether or not survival \emph{ever} occurs when $d=1$.
See Section \ref{rem:6}.
The above theorems are proved using a perturbative argument, and thus fall short of the assertion that 
$\ula=\lac$. 

The methods of proof may be made quantitative, leading to bounds for
the numerical values of the critical points $\ac$. Such bounds are far from precise,
and therefore we do not explore them here. Our basic estimates for the growth
of infection hold if the intensity $\la$ of the Poisson process
is non-constant so long as it is bounded uniformly 
between two strictly positive constants. The existence
 of the subcritical phase may be proved
for more general diffusions than Brownian motion.

\subsection{Literature and notation}\label{ssec1-3}
The related literature is somewhat ramified, and a spread of related 
problems have been studied by various teams. We mention a selection of papers but do not attempt a full review,
and we concentrate on works associated with $\RR^d$ rather than with trees or complete graphs.

The delayed diffusion model may be viewed as a continuous-time version of the 
\lq frog' random walk process studied in 
Alves et al.\ \cite{AMP02a,AMP02b}, Ramirez and Sidoravicius \cite{RS02},
Fontes et al.\  \cite{FMS}, Benjamini et al.\ \cite{BFHM}, and Hoffman, Johnson, and Junge \cite{HJJ,JJ}.
See Popov \cite{Pop03} for an early review.
Kesten and Sidoravicius \cite{KS05,KS06, KS3} considered a variant of the frog model as a model for infection, both 
with and without recuperation (that is, when infected frogs recover and become available for reinfection---see
also Section \ref{rem:6} of the current work).  
The paper of Beckman et al.\ \cite{BDDHJ} is devoted to the delayed diffusion model without removal
(that is, with $\alpha=0$). Peres et al.\  \cite{psss12} studied three geometric properties of 
a Poissonian/Brownian cloud of particles,
in work inspired in part by the dynamic Boolean percolation model of van den Berg et al.\ \cite{BMW}.
Related work has appeared in Gracar and Stauffer \cite{GS18}.

A number of authors have considered the frog model with recuperation under the title \lq activated random walks'. 
The reader is referred to the review by Rolla \cite{Rolla15}, and for recent work to 
Stauffer and Taggi \cite{ST18} and Rolla et al.\ \cite{RSZ}.

We write $\ZZ_0=\{0,1,2,\dots\}$ and $1_A$  (or $1(A)$)
for the indicator function of an event or set $A$.
Let $S(r)$ denote the closed $r$-ball of $\RR^d$ with centre at the origin, and $S=S(1)$.
The $d$-dimensional Lebesgue measure of a set $A$ is written $|A|_d$,
and the Euclidean norm $\|\cdot\|_d$. The \emph{radius} of $M \subseteq \RR^d$ is defined by
$$
\rad(M):= \sup\{\|m\|_d: m\in M\}.
$$
We abbreviate $\Pda$ (\resp, $\Eda$) to the generic notation $\PP$ (\resp, $\EE$).

The contents of this paper are as follows. The two models are defined in Section \ref{sec:gen}
with a degree of generality that includes general diffusions and a more general process of infection. 
The delayed diffusion model is studied in Section \ref{sec:dd}, and the diffusion model in Section \ref{sec:d}.
Theorem \ref{thm:1} (\resp, Theorem \ref{thm:2})
is contained within Theorem \ref{thm:1x} (\resp, Theorem \ref{thm:2xx}).

\subsection{Open problems}\label{ssec1-4}

This introduction closes with a short account of some of the principal remaining open problems.
For concreteness, we restrict ourselves to the Brownian models of Section \ref{ssec1-2}
without further reference to the random-walk versions of these models, and the general models of Section \ref{sec:3-1-0}.
This section is positioned here despite the fact that it refers sometimes to versions of the models
that have not yet been fully introduced (see Section \ref{sec:gen}). 

\begin{Alist}
\item For the Brownian delayed diffusion model, show that the critical value $\ac(\la)$ of
Theorem \ref{thm:1} satisfies $\ac(\la) <\oo$ whenever $\la <\lac$.

\item When $d\ge 2$, prove survival in the Brownian diffusion model for some $\la<\lac$ and small 
$\a>0$.
More specifically, show that $\thd(\la,\a)>0$ for suitable $\la$ and $\a$.

\item Having resolved problem B, show the existence of a critical value $\ac=\ac(\la)$ 
for the Brownian diffusion model such that
survival occurs when $\a<\ac$ and not when $\a>\ac$. Furthermore, identify the set of $\la$ such that
$\ac(\la)<\oo$.

\item Decide whether or not survival can ever occur for the Brownian diffusion model in one dimension.

\item In either model, prove a shape theorem for the set of particles that are either infected or removed at time $t$.
 
\end{Alist}

\section{General models}\label{sec:gen}

\subsection{The general set-up}\label{sec:3-1-0}
Let $d \ge 1$. A \emph{diffusion process} in $\RR^d$ is a solution $\zeta$ to the
stochastic differential equation
\begin{equation}\label{sde}
d\zeta(t)= a(\zeta(t))\,dt + \si(\zeta(t))\,dW_t,
\end{equation}
where $W$ is a standard Brownian motion in $\RR^d$. 
(We may write either $W_t$ or $W(t)$.)
For definiteness, we shall assume that: $\zeta(0)=0$;
$\zeta$ has continuous sample paths; the instantaneous drift vector $a$ and 
variance matrix $\si$ are locally Lipschitz continuous. We do not allow
$a$, $\si$ to be time-dependent.
We call the process \lq Brownian' if $\zeta$ is a standard Brownian motion, which is to
say that $a$ is the zero vector and $\si$ is the identity matrix.

Let $\zeta$ be such a diffusion, and
let $(\zeta_i:i \in\ZZ_0)$ be independent copies of $\zeta$.
Let $\alpha\in(0,\oo)$, $\rho\in[0,\oo)$, and let  $\mu:\RR^d\to[0,\oo)$
be integrable with 
\begin{equation}\label{mu-int}
\iNT(\mu) := \int_{\RR^d} \mu(x)\,dx  \in (0,\oo).
\end{equation} 
We call $\mu$ \emph{radially decreasing} if
\begin{equation}\label{eq:si}
\mu(rx) \le \mu(x) \qq x\in\RR^d,\, r \in[1,\oo). 
\end{equation}

Let $\Pi=(X_0=0, X_1, X_2,\dots)$ be a  
Poisson process on $\RR^d$ (conditioned to possess a point at the origin $0$)
with constant  intensity 
$\la \in (0,\oo)$.
At time $0$, particles with label-set $\sP=\{P_0,P_1,P_2,\dots\}$
are placed at the respective points $X_0=0,X_1,X_2,\dots$. 
We may refer to a particle $P_i$ by either its index
$i$ or its initial position $X_i$.

We describe the process of infection in a somewhat informal manner (see also Section \ref{sec:con}).
For $i\in\ZZ_0$, at any given time $t$ particle $P_i$ is in one of three states S (susceptible),
I (infected), and R (removed). Thus the state space is $\Om=\Pi\times\{\Ss,\Ii,\Rr\}^{\ZZ_0}$,
and we write $\om(t)=(\om_i(t): i \in\ZZ_0)$ for the state of the process
at time $t$.
Let $S_t$ (\resp, $I_t$, $R_t$) be the set of
particles in state S (\resp, I, R) at time $t$. We take 
$$
\om_i(0)=\begin{cases} \Ii &\text{if } i = 0,\\
\Ss &\text{otherwise},
\end{cases}
$$
so that $I_0=\{P_0\}$ and $S_0=\sP\setminus\{P_0\}$. 
The only particle-transitions that may occur are S $\to$ I and
I $\to$ R.
The transitions $\Ss\to\Ii$ occur at rates that
depend on the locations of the currently infected particles.

We shall refer to the above (in conjunction with the 
specific infection assumptions of Sections \ref{sec:3-1} or \ref{sec:3-1c}) 
as the \emph{general model}. When $a\equiv 0$ and $\si\equiv 1$ in \eqref{sde}
(or, more generally, $\si$ is constant), we shall refer to it as the \emph{Brownian model}. 
We shall prove the existence of a subcritical phase (characterized by the absence of survival) for the general model
subject to weak conditions.
Our proof of survival for the delayed diffusion model is for the Brownian model alone. 
Estimates for the volume of the sausage generated by $\zeta$ play roles in the calculations, 
and it may be that, in this regard or another, the behaviour of a general model is richer than
that of its Brownian version. 

\subsection{Delayed diffusion model}\label{sec:3-1}

Each particle $P_j$ is stationary if and only if it is in state S. If it becomes infected
(at some time $B_j$, see \eqref{inf-tm}), 
henceforth it follows the diffusion $X_j+\zeta_j$. We write
\begin{align*}
\pi_j(t) = \begin{cases} X_j&\text{if } t\le B_j,\\ 
X_j+\zeta_j(t-B_j) \q&\text{if } t>B_j,
\end{cases}
\end{align*}
for the position of $P_j$ at time $t$.

We describe next the rate at which a given particle $P$ infects another particle $P'$. 
The function $\mu$, given above,
encapsulates the spatial aspect of the infection process, and a parameter $\rho\in(0,\oo)$
represents its intensity,  
\begin{letlist}
\item[($\Ss\to\Ii$)] Let $t>0$, and 
let $P_j$ be a particle that is in state $\Ss$ at all times $s<t$.
Each $P_i\in I_t$ (with $i\ne j$) infects $P_j$ at rate
$\rho\mu(X_j-\pi_i(t))$. The
aggregate rate at which $P_j$ becomes infected is
\begin{equation}\label{eq:inf-r}
\sum_{i\in I_t,\, i\ne j} \rho \mu(X_j-\pi_i(t)).
\end{equation}

\item[($\Ii\to\Rr$)] An infected particle is removed at rate $\alpha$.
\end{letlist}
Transitions of other types are not permitted. 
We take the sample path $\om=(\om(t):t \ge 0)$ to
be pointwise right-continuous, which is to say that, for $i\in\ZZ_0$, the function
$\om_i(\cdot)$ is right-continuous.
The \emph{infection time} $B_j$ of particle $P_j$ is given by
\begin{equation}\label{inf-tm}
B_j=\inf\{t\ge 0: P_j\in I_t\}.
\end{equation}

The infection rates $\rho\mu(X_j-\pi_i(t))$ of \eqref{eq:inf-r}
are finite, and hence infections take place at a.s.\ distinct times. We may thus speak of
$P_j$ as being \lq directly infected' by $P_i$.
We speak of a point $z\in\Pi$ as being \emph{directly infected} by a point $y\in\Pi$
when the associated particles have that property.  If $P_j$ is infected directly by $P_i$, we call $P_j$  a \emph{child}
of $P_i$, and $P_i$ the \emph{parent} of $P_j$.

Following its infection, particle $P_i$ remains infected for a 
further random time $T_i$, called the \emph{lifetime} of $P_i$,
 and
is then removed. The times  $T_i$ are random variables with 
the exponential distribution with parameter $\alpha > 0$, and 
are independent of one another and of the $X_j$ and $\zeta_j$.

In the above version of the delayed diffusion model, 
$\rho$ is assumed finite. When $\rho=\oo$, we shall consider only situations in which
\begin{equation}\label{eq:inftyass}
\rho=\oo, \qq \mu=1_M \text{ where $M\subseteq \RR^d$ is compact.}
\end{equation}
In this situation, a susceptible particle $P_j$ becomes 
infected at the earliest instant that it belongs to $\pi_i(t)+M$ for some $P_i\in I_t$,
$i \ne j$.
This happens when either (i) $P_i$ infects $P_j$ as $P_i$ diffuses around $\RR^d$ post-infection,
or (ii)  at the moment $B_i$ of infection of $P_i$, particle $P_j$ is infected instantaneously by virtue
of the fact that $X_j\in \pi_i(B_i)+M$. These two situations are investigated slightly more fully
in the following definition of \lq direct infection'.

The role of the Boolean model of continuum percolation becomes clear when $\rho=\oo$,
and we illustrate this, subject to the simplifying assumption
that $M$ is symmetric in the sense
that $x\in M$ if and only if $-x\in M$.
Let $\Pi=(X_i: i \in\ZZ_0)$ be a Poisson process in $\RR^d$ with constant intensity $\la$,
and declare two points $X_i$, $X_j$ to be \emph{adjacent} if and only if $X_j-X_i\in  M$. 
This adjacency relation generates a graph $G$ with vertex-set $\Pi$.
In the delayed diffusion process on the set $\Pi$, entire clusters of the percolation 
process are infected simultaneously.

Since there can be many (even infinitely many)
simultaneous infections at the same time instant when $\rho=\oo$, the notion of \lq direct infection' 
requires amplification.
For $j \ne 0$, we say that $P_j$ is \emph{directly infected} by $P_0$ if
$P_j$ is in state S at all times $s<Y_j$, where $Y_j=\inf\{t \ge 0: X_j\in \zeta_0(t)+M\}$,
and in addition $Y_j < T_0$. 
We make a similar definition, as follows,  for direct infections by $P_i$ with $i\ne 0$. 
Let $i\ne 0$ and $j\ne 0,i$. 
\begin{letlist}
\item We say that $P_j$ is \emph{dynamically infected} by $P_i$ if the following holds.
Particle $P_i$ (\resp, $P_j$) is in state I (\resp,  state S) at all times $Y_{i,j}-\eps$
for $\eps\in(0,\eps_0)$ and some $\eps_0>0$, where 
$Y_{i,j}=\inf\{t \ge B_i: X_j\in X_i+\zeta_i(t)+M\}$,
and in addition $Y_{i,j}-B_i \le T_i$.
\item  We say that $P_j$ is \emph{instantaneously infected} by $P_i$ if the following holds.
There exist $n\ge 1$ and $k, i_1,i_2,\dots,i_n=i, i_{n+1}=j$ such that 
$i_1$ is dynamically infected by $k$ (at time $B_{i_1}$) and
$$
X_{i_{m+1}} \in T_m\sm T_{m-1}, \qq m=1,2,\dots,n,
$$
where
$$
T_0:= \es, \q T_m=\bigcup_{r=1}^m (X_r+M).
$$
\end{letlist}
Condition (a) corresponds to infection through movement of $P_i$, and (b) corresponds to instantaneous
infection at the moment of infection of $P_i$. We say that $P_j$ is \emph{directly infected} by $P_i$ if
it is infected by $P_i$ either dynamically or instantaneously.

Certain events of probability $0$ are overlooked in the above informal description including, for example,
the event of being dynamically infected by two or more particles, 
and the event of being instantaneously infected by an infinite chain $(i_r)$ but by no
finite chain.

In either case $\rho<\oo$ or $\rho=\oo$, we write
$\thdd(\la,\rho,\alpha)$ for the probability that infinitely many particles are infected.
For concreteness, we note our special interest in the case in which:
\begin{letlist}

\item $\zeta$ is a standard Brownian motion,

\item $\mu=1_S$ with $S$ the closed unit ball of $\RR^d$.

\end{letlist}

\subsection{Diffusion model}\label{sec:3-1c}

The diffusion model differs from the delayed diffusion model
of Section \ref{sec:3-1} in that all particles begin to move at time $t=0$.
The location of $P_j$ at time $t$ is $X_j+\zeta_j(t)$, and the transition rates are
given as follows. First, suppose $\rho\in(0,\oo)$.
\begin{letlist}
\item[($\Ss\to\Ii$)] Let $t>0$, and 
let $P_j$ be susceptible at all times $s<t$.
Each $P_i\in I_t$ (with $i\ne j$) infects $P_j$ at rate
$\rho\mu(X_j+\zeta_j(t)-X_i-\zeta_i(t))$. The
aggregate rate at which $P_j$ becomes infected is
\begin{equation}\label{eq:inf-r2}
\sum_{i\in I_t,\, i\ne j} \rho \mu\bigl(X_j+\zeta_j(t)-X_i-\zeta_i(t)\bigr).
\end{equation}

\item[($\Ii\to\Rr$)] An infected particle is removed at rate $\alpha$.
\end{letlist}

As in Section \ref{sec:3-1}, we may allow $\rho=\oo$ and $\mu=1_M$ 
with $M$ compact. 
In either case $\rho<\oo$ or $\rho=\oo$ we write
$\thd(\la,\rho,\alpha)$ for the probability that infinitely many particles are infected.

\subsection{Construction}\label{sec:con}

We shall not investigate the formal construction of the above processes as strong Markov processes
with right-continuous sample paths. The interested reader may refer to the
related model involving random walks on $\ZZ^d$ (rather than diffusions or Brownian motions
on $\RR^d$) with $\alpha=0$,
as considered in some depth by Kesten and Sidoravicius in 
\cite{KS05} and developed for the process with \lq recuperation' in their sequel \cite{KS06}.

Instead, we sketch briefly how such processes may be built around a triple  $(\Pi,\bzeta,\bde)$,
where $\Pi$ is a rate-$\la$ Poisson process of initial positions,
$\bzeta=(\zeta_i:i \in\ZZ_0)$ is a family of independent copies of the diffusion $\zeta$, and  
$\bde=(\de_i: i \in\ZZ_0)$ is a family
of independent rate-$\alpha$ Poisson processes on $(0,\oo)$, that are independent 
of the pair $(\Pi, \bzeta)$.  We place particles $P_i$ at the points of $\Pi$, and 
$P_i$ deviates from its initial point according to $\zeta_i$. An initial infected particle $P_0$
is placed at the origin at time $0$, and it diffuses according to $\zeta_0$. The infection is communicated
according to the appropriate rules (either delayed or not, and 
 either (i)  via  the pair $(\rho,\mu)$ with $\rho<\oo$, or (ii) with $\rho=\oo$
and $\mu=1_S$).  After infection, $P_i$ is removed at the next occurrence of the Poisson process $\de_i$.

The above construction is straightforward so long as there exist, at any given time a.s., only
finitely many simultaneous infections. 
In the two models with $\rho<\oo$, simultaneous infections can occur only after the earliest time $T_\oo$ at
which there exist infinitely many infected particles. It is a consequence of Proposition \ref{prop:2-1}(c)
that $\PP(T_\oo<\oo)=0$ for the general delayed diffusion model with $\rho<\oo$.

The issue is slightly more complex when $\rho=\oo$ and $\mu=1_S$, since infinitely many
simultaneous infections may take place in the supercritical phase of the percolation 
process of moving disks. 
In this case, we assume invariably that $\la<\lac$, so that there is no percolation
of disks at any fixed time, and indeed it was shown in \cite{BMW} that, a.s.,  there is no percolation
at all times. Therefore, there exist, a.s., only finitely many simultaneous infections at any given instant.
Bounds for the growth of generation sizes are found at \eqref{g7-1zz} and \eqref{g90}.

\section{The delayed diffusion model}\label{sec:dd}

\subsection{Main results}\label{sec:dd-main}

We consider the Brownian delayed diffusion model of Section \ref{sec:3-1},
and we adopt the notation of that section. Recall the critical 
point $\lac$ of the Boolean continuum percolation on $\RR^d$ in which a
closed unit ball is placed at each point
of a rate-$\la$ Poisson process. Let $\thdd(\la,\rho,\a)$ be the probability that the process survives.

\begin{theorem}\label{thm:1x}
Consider the Brownian delayed diffusion model on $\RR^d$ where $d \ge 2$.
\begin{letlist}

\item Let $\rho\in(0,\oo)$. There exists a function 
$\ac:(0,\oo)^2\to (0,\oo)$ such that 
\begin{equation}\label{g4b}
\thdd(\la,\rho,\alpha) \begin{cases} 
=0 &\text{if } \alpha>\ac(\la,\rho),\\
>0 &\text{if } \alpha<\ac(\la,\rho).
\end{cases}
\end{equation}
The function $\thdd(\la,\rho,\alpha)$ is non-increasing in $\alpha$
and non-decreasing in $\rho$. Therefore, $\ac=\ac(\la,\rho)$ is non-decreasing in $\rho$.

\item Let $\rho=\oo$ and $\mu=1_S$ where 
$S$ is the closed unit ball in $\RR^d$.
There exists a non-decreasing function 
$\ac:(0,\oo)\to (0,\oo]$ such that, for $0<\la<\lac$,
\begin{equation}\label{g4bb}
\thdd(\la,\oo,\alpha) \begin{cases} 
=0 &\text{if } \alpha>\ac(\la),\\
>0 &\text{if } \alpha<\ac(\la).
\end{cases}
\end{equation}
Furthermore, there exists $\ula\in(0,\lac]$ such that  
\begin{equation*}
\ac(\la) \begin{cases}
<\oo &\text{if } 0<\la<\ula,\\
=\oo &\text{if } \la>\lac.
\end{cases}
\end{equation*}
Moreover, the function $\thdd(\la,\oo,\alpha)$ is non-increasing in $\alpha$. 
\end{letlist}
\end{theorem}

This theorem extends Theorem \ref{thm:1}. Its proof is found in Sections \ref{sec:prel}--\ref{sec:3}. 

The situation is different in one dimension, where it turns out that the Brownian model has no phase transition. 
The proof of the following theorem, in a version valid for the general delayed diffusion model, may be found 
in Section \ref{sec:nsur1d}. 

\begin{theorem}\label{thm:nopt}
Consider the Brownian delayed diffusion model on $\RR$ with $\mu=1_S$.
We have that $\thdd(\la,\rho,\alpha)=0$  for all $\la,\alpha>0$ and $\rho\in(0,\oo]$.  
\end{theorem}

Theorems \ref{thm:1x} and \ref{thm:nopt} are stated for the case of a single initial infective.
The proofs are valid also with a finite number of initial infectives distributed at
the points of some arbitrary subset $I_0$ of $\RR^d$. By the proof of the forthcoming Proposition \ref{prop:4}, 
the set $I$ of ultimately infected particles is
stochastically increasing in $I_0$.  

\subsection{Percolation representation of the delayed diffusion model}\label{sec:prel}
Consider the delayed diffusion model with $d\ge 2$.
Suppose that either $\rho\in(0,\oo)$ with $\mu$ as in \eqref{mu-int}, or $\rho=\oo$ and 
\begin{equation}\label{g20}
\mu(x)= 1_S(x) ,\qq x \in \RR^2,
\end{equation}
where $S$ is the closed unit ball with centre at the origin. It turns out that the set of infected 
particles may be considered as a type of percolation model on the random set $\Pi$. This observation
will be useful in exploring the phases of the former model.

The proof of the main result of this section, Proposition \ref{lem:3}, is motivated in part by work of Kuulasmaa \cite{KK}
where a certain epidemic model was studied via a related percolation process
(a similar argument is implicit in \cite[p.\ 4]{AMP02a}).
Recall the initial placements $\Pi=(X_0=0,X_1,X_2,\dots)$ of particles $P_i$,
with law denoted $\Po$ (and corresponding expectation $\Eo$); we  
condition on $\Pi$. 

Fix $i\ge 0$, and consider the following infection process.  
The particle  $P_i$ is the \emph{unique} 
initially infected particle, and it diffuses according to $\zeta_i$ and has lifetime $T_i$.
All other particles $P_j$, $j \ne i$, are kept stationary for all time at their respective
locations $X_j$. As $P_i$ moves around $\RR^d$, it infects other particles in the usual way; 
newly infected particles are permitted neither to move nor to infect others.
Let $J_i$ be the (random) set of particles infected by $P_i$
in this process. 

Let $\tau_{i,j}\in (0,\oo]$ be the time of the first infection by $P_i$
of $P_j$, \emph{assuming that $P_i$ is never removed}.
Write $i \to j$
if $\tau_{i,j}< T_i$, which is to say that this infection takes place before
$P_i$ is removed.  Thus, 
\begin{equation}\label{eq:defJ}
J_i=\{j: i\to j\}.
\end{equation}

\emph{Suppose first that $\rho<\oo$.}
Given $(\Pi, \zeta_i, T_i)$, the vector $\vtau_i=(\tau_{i,j}: j\ne i)$ contains conditionally independent
random variables with respective distribution functions
\begin{equation}\label{g94}
F_{i,j}(t)  = 1-\exp\left(-\int_0^t \rho\mu(X_j-X_i-\zeta_i(s))\,ds\right),\qq t\ge 0,
\end{equation}
and
\begin{equation}\label{g94c}
\PP(i \to j\mid \Pi, \zeta_i, T_i) =F_{i,j}(T_i). 
\end{equation}

\emph{When $\rho=\oo$}, we have that 
\begin{equation}\label{g95}
\tau_{i,j}=\inf\{t>0: X_j\in X_i+\zeta_i(t)+S\},
\end{equation}
the first hitting time of $X_j-X_i$ by the radius-$1$ sausage of $\zeta_i$.
As above, we write $i \to j$ if $\tau_{i,j} < T_i$, with $J_i$ and $\vtau_i$ given accordingly.

One may thus construct sets $J_i$ for all $i \ge 0$; given $\Pi$,  the set $J_i$ 
depends only on $(\zeta_i,T_i)$, and therefore the $J_i$ 
are conditionally independent given $\Pi$.
The sets $\{J_i: i \ge 0\}$ generate a directed graph $\vec G=\vec G_\Pi$ with vertex-set $\ZZ_0$
and directed edge-set $\vec E=\{[i,j\rangle: i \to j\}$. Write $\vec I$ for the set of
vertices $k$ of $\vec G$ such that there exists a 
directed path of $\vec G$ from $0$ to $k$.
To the edges of $\vec G$ we attach random labels, with 
edge $[i,j\rangle$ receiving the label $\tau_{i,j}$. 

From the vector $(\vtau_i, T_i: i \ge 0)$, we can construct a copy of the general delayed diffusion process
by allowing  an infection by $P_i$ of $P_j$ whenever $i\to j$ and in addition
$P_j$ has not been infected previously by another particle.
Let $I$ denote the set of ultimately infected particles in this coupled process.

\begin{proposition}\label{lem:3}
For $\rho\in(0,\oo]$, we have $I=\vec I$.
\end{proposition} 

We turn our attention to the Brownian case. 
By rescaling in space/time, we obtain the following. The full parameter-set of
the process is $\{\la, \rho, \alpha, \mu, \si\}$, where $\si$ is the standard-deviation parameter of
the Brownian motion, 
and we shall sometimes write $\thdd(\la,\rho,\alpha,\mu,\si)$ accordingly. 

\begin{proposition}\label{prop:4}
Consider the Brownian delayed diffusion model.
Let $\rho\in(0,\oo]$. 
\begin{letlist}
\item
For given $\la\in(0,\oo)$, the function $\thdd(\la,\rho,\alpha)$ is
non-decreasing in $\rho$ and non-increasing in $\alpha$. 
\item We have that
\begin{equation}\label{g94d}
\thdd(\la,\rho,\alpha,\mu,1)= \thdd(\la/r^d, \rho/r^2,\alpha/r^2,\mu_r,1), \qq r \ge 1,
\end{equation}
where $\mu_r(x):=\mu(x/r)$. 
\item If $\mu$ is radially decreasing  (see \eqref{eq:si}), then 
$$
\ac(\la,\rho)  \ge r^2 \ac(\la/r^d,\rho/r^2), \qq r \ge 1.
$$
\item If $\rho=\oo$ and $\mu$ is radially decreasing, then
$\thdd(\la,\oo,\alpha)$ and $\ac(\la,\oo)$ are non-decreasing in $\la$.
\end{letlist}
\end{proposition}

\begin{proof}[Proof of Proposition \ref{lem:3}]
This is a deterministic claim.
Assume $\Pi$ is given.
If $i\in I$, there exists a chain of 
direct infection from $0$ to $i$, 
and this chain generates a directed path of $\vec G$
from $0$ to $i$. 
Suppose, conversely, that $k\in \vec I$. 
Let $\sP_k$ be the set of directed paths of $\vec G$  from $0$ to $k$.
Let $\pi\in\sP_k$ be a shortest such path (where the length of an edge $[i,j\rangle$ is
taken to be the label $\tau_{i,j}$ of that edge). 
We may assume that the $\tau_{i,j}$, for $i\to j$, 
are distinct; no essential difficulty emerges on the complementary null set.
Then the path $\pi$ is a geodesic, in that every sub-path is 
the shortest directed path joining its
endvertices. Therefore, when infection is initially introduced at $P_0$,
it will be transmitted directly along $\pi$ to $P_k$. 
\end{proof}

\begin{proof}[Proof of Proposition \ref{prop:4}]

(a) By Proposition \ref{lem:3}, if the parameters are changed in such a way that 
each $J_i$ is stochastically increased (\resp, decreased), then the set $I$ is also stochastically 
increased (\resp, decreased). The claims follow by \eqref{g94}--\eqref{g94c} when $\rho<\oo$,
and by \eqref{g95} when $\rho=\oo$.

(b)  
We shall show that the probabilities of infections are the same for the two sets of parameter-values in \eqref{g94d}.
Let $r \ge 1$, and consider the effect of dilating space by the ratio $r$. The resulting stretched Poisson process $r\Pi$ has intensity $\la/r^d$, the resulting Brownian motion $r\zeta_i(t)$
is distributed as $\zeta_i(r^2t)$, and $\mu$ is replaced by $\mu_r$. Therefore,
\begin{equation}\label{g94e}
\thdd(\la,\rho,\alpha,\mu,1) = \thdd(\la/r^d,\rho,\alpha,\mu_r,r).
\end{equation}

Next, we use the construction of the process in terms of the $J_i$ given above Proposition \ref{lem:3}. 
If $\rho<\oo$ then, by \eqref{g94c} and the change of variables $u=r^2s$,
\begin{align*}
\PP(i\to j\mid \Pi, \zeta_i,T_i) &= 1-\exp\left(-\int_0^{T_i}\rho\mu(X_j-X_i-\zeta_i(s))\,ds\right)\\
&\eqd 1-\exp\left(-\int_0^{T_i}\rho\mu_r(rX_j-rX_i-\zeta_i(r^2s))\,ds\right)\\
&\eqd 1-\exp\left(-\int_0^{r^2T_i}\rho\mu_r(rX_j-rX_i-\zeta_i(u))\,\frac{du}{r^2}\right),
\end{align*}
where $\eqd$ means equality in distribution.
Since $r^2 T_i$ is exponentially distributed with parameter $\alpha/r^2$,
the right side of \eqref{g94e} equals $\thdd(\la/r^d,\rho/r^2,\alpha/r^2,\mu_r,1)$, as claimed. 
The same conclusion is valid for $\rho=\oo$, by \eqref{g95}.

(c)  
Since $\mu_r\ge \mu$ by assumption, the $J_i$ are stochastically monotone in $\mu$,
it follows by \eqref{g94d} that 
$$
\thdd(\la,\rho,\alpha,\mu,1)\ge \thdd(\la/r^d, \rho/r^2,\alpha/r^2,\mu,1),\qq r\ge 1.
$$
By the monotonicity of $\thdd$ in $\alpha$, if $\alpha>\ac(\la,\rho)$ then
$\alpha/r^2\ge \ac(\la/r^d,\rho/r^2)$ as claimed.

(d)
This holds as in part (c).

\end{proof}

\begin{remark}\label{rem:2}
In the forthcoming proof of Section \ref{ssec:dge3} we shall make use of a consequence of Proposition \ref{lem:3},
namely that 
\begin{equation}\label{g93}
\thdd(\la,\rho,\alpha) = \Eo\bigl(\QQ_\Pi(|\vec I|=\oo)\bigr),
\end{equation}
where $\QQ_\Pi$ is the conditional law of $\vec G$ given $\Pi$, and $\Eo$ is expectation with respect to $\Pi$.
In proving survival, it therefore suffices to prove the right side of \eqref{g93} is
strictly positive. 
\end{remark}

\subsection{No survival in one dimension}\label{sec:nsur1d}

It was stated in Theorem \ref{thm:nopt} that the Brownian model 
with $\mu=1_S$
never survives in one dimension. 
We state and prove a version of this \lq no survival' theorem for the general delayed diffusion model 
of Section \ref{sec:3-1}, subject to a weak condition which includes the Brownian model.

Throughout this section, $T_\alpha$ denotes a random variable 
having the exponential distribution with parameter $\alpha$, assumed to be independent of
all other random variables involved in the models. A typical diffusion is denoted $\zeta$, and we write
$M_t=\sup\{|\zeta(s)|: s\in [0,t]\}$.

\begin{theorem} \label{thm:nopt+}
Consider the general delayed diffusion model
on $\RR$  with infection parameters $(\rho,\mu)$. 
\begin{letlist}

\item Let $\rho<\oo$. Assume that 
$\alpha$ is such that $\EE|\zeta(T_\alpha)| <\oo$, and in addition that $\int_\RR |y|\mu(y)\,dy<\oo$.
Then $\thdd(\la,\rho,\alpha)=0$  for all $\la>0$.  

\item Let $\rho=\oo$ and $\mu$ have bounded support. Assume that $\alpha$ is such that 
$\EE(M_{T_\alpha})<\oo$. 
Then $\thdd(\la,\oo,\alpha)=0$  for all $\la>0$.
\end{letlist}
\end{theorem}

It follows that. for all $\alpha$, there is no survival if
\begin{align*}
\text{either:}\q&\text{$\rho<\oo$ and, for all $\alpha$, we have  
$\EE|\zeta(T_\alpha)| <\oo$,}\\
\text{or:}\q&\text{$\rho=\oo$, $\mu$ has bounded support, and, for all $\alpha$, we have $\EE(M_{T_\alpha}) <\oo$}.
\end{align*}
These two conditions (that $\EE|\zeta(T_\alpha)| <\oo$ and $\EE(M_{T_\alpha}) <\oo$)
are equivalent when $\zeta$ is Brownian motion
(see, for example, \cite[Thm 13.4.6]{GS20}), and indeed they hold for all $\alpha$ in the Brownian case.
This implies Theorem \ref{thm:nopt}. The proof of Theorem \ref{thm:nopt+} has some similarity to that
of \cite[Thm 1.1]{AMP02a}. 

\begin{proof}
(a) Assume the required conditions. 
Let  $A>0$; later we will take $A$ to be large. Let $\Pi'$ be a Poisson process on $\RR$ with intensity $\la$,
and let $L=\Pi'\cap (-\oo,0]$ and $R=\Pi'\cap [A,\oo)$. Write $\{S_1 \to S_2\}$ for the event that,
in the percolation representation of the last section,  some particle in $S_1$ infects some particle in $S_2$.
We prove first that
\begin{equation}\label{eq:new500}
\PP(L\to R) \to 0\qq\text{as } A \to \oo.
\end{equation}

Note that, for suitable functions $f$,
\begin{equation}\label{eq:new501}
\EE\left(\int_0^{T_\alpha} f(\zeta(s))\, ds\right ) = \int_0^\oo \EE(f(\zeta(s)) e^{-\alpha s} \, ds
= \frac1\alpha \EE(f(\zeta(T_\alpha)). 
\end{equation}
Since $\PP(L\to R)$ is no larger than the mean number of infections from $L$ into $R$, 
we have by the Campbell--Hardy Theorem for Poisson processes (see \cite[Exer.\ 6.13.2]{GS20}), 
Fubini's Theorem, 
and \eqref{eq:new501} that
\begin{align}\label{eq:new502}
\PP(L \to R) & \le \la^2\int_{-\oo}^0 du \int_A^\oo dv\, \EE\left(\rho\int_0^{T_\alpha} \mu(v-u-\zeta_u(s))\,ds\right)\\
&= \frac{\la^2\rho}\alpha \EE\left( \int_{-\oo}^0 du \int_A^\oo dv\, \mu(v-u-\zeta(T_\alpha))\right)\nonumber\\
&= \frac{\la^2\rho}\alpha \EE\left(\int_A^\oo dv\, I(v-\zeta(T_\alpha))\right)
= \frac{\la^2\rho}\alpha \EE(Z_A), \nonumber
\end{align}
where
\begin{equation*}
I(x) := \int_x^\oo \mu(y)\,dy, \qq Z_A:= \int_{A-\zeta(T_\alpha)}^\oo I(v)\, dv.
\end{equation*}

Since $I(x) \le I(-\oo)<\oo$, we have
$$
\EE(Z_A) \le  I(-\oo)\EE|\zeta(T_\alpha)| + \int_A^\oo I(v)\, dv.
$$
Furthermore,  $Z_A$ is integrable since
$$
\int_A^\oo I(v)\, dv \le \int_0^\oo I(v)\, dv = \int_0^\oo y\mu(y)\,dy <\oo.
$$
Since $Z_A\to 0$ a.s.\ as $A\to\oo$,  we have by monotone convergence that
$\EE(Z_A) \to 0$ also. Equation \eqref{eq:new500} now follows by \eqref{eq:new502}.
By a similar argument, $\PP(R \to L)\to 0$ as $A\to\oo$.

We pick $A$ sufficiently large that
$$
\PP(L \to R) \le \tfrac14, \qq \PP(R \to L) \le \tfrac14.
$$
On the event $E_A:=\{L\not\to R\}\cap \{R \not\to L\} \cap\{\Pi'\cap(0,A) = \es\}$, 
there can be no chain of infection 
between particles in $L$ and particles in $R$.
Note that
\begin{equation}\label{eq:new506}
\PP(E_A) \ge \tfrac12 e^{-\la A}>0.
\end{equation}

Let $B_0=E_A$ and, for $k\in \ZZ$, let $B_k$ be the event defined similarly to $E_A$ but with the interval
$(0,A)$ replaced by $(kA,(k+1)A)$. By \eqref{eq:new506},
\begin{equation}\label{g105}
\PP(B_k) = \PP(E_A) \ge \tfrac12 e^{-\la A} >0.
\end{equation}
By the ergodic theorem, the limit
$$
\La:= \lim_{n\to\oo} \frac1{2n+1} \sum_{k=-n}^n 1_{B_k} 
$$
exists a.s.\ and has mean at least $\frac12 e^{-\la A}$. 
Since $\La$ is translation-invariant and the underlying probability measure is a product measure,
we have $\La\ge \tfrac13 e^{-\la A}$ a.s. Therefore, $B_k$ occurs infinitely often a.s.
This would imply the claim were it not for the extra particle at $0$ and its associated Brownian motion $\zeta_0$. 
However, the range of $\zeta_0$, up to time $T_0$, is a.s.\ bounded, and this completes the proof.

(b) Let $\rho=\oo$ and, for simplicity, take $\mu=1_S$ (the proof for general $\mu$ with bounded support 
is essentially the same). The proof is close to that of part (a).

Let $\Pi'$ be a Poisson process on $\RR$ with intensity $\la$,
let $(\zeta_X: X\in \Pi')$ be independent copies of $\zeta$, and let $(T_X:X\in\Pi')$ be independent
copies of $T_\alpha$. 
Write $F_x=\inf\{t\ge 0: \zeta(t)=x\}$
for the first-passage time of  $\zeta$ to the point $x\in\RR$.

Let 
$$
A = \bigcap_{X\in \Pi'\cap(-\oo,0)} \{G_X>T_X\},
$$ 
where $G_X$ is the first-passage time to $0$ of $X+\zeta_X$. Then
\begin{equation}\label{g101}
\PP(A) \ge \PP\bigl(\Pi'\cap[-1,0]=\es\bigr) 
\EE\left(\prod_{X\in \Pi'\cap(-\oo,-1)} (1-p_X)\right),
\end{equation}
where 
\begin{equation}\label{g101b}
p_x=\PP(G_x \le T_x)=\PP(F_{-x}\le T_\alpha) \le \PP(M_{T_\alpha}\ge |x|).
\end{equation}
There exists $\eps>0$ such that $p_x<1-\eps$ for $x \ge 1$, 
and therefore there exists $c=c(\eps)\in(0,\oo)$ 
such that $1-p_x\ge e^{-cp_x}$ for $x \ge 1$. By \eqref{g101} and Jensen's inequality,
\begin{equation}\label{g102}
\PP(A) \ge e^{-\la}\exp\left(-c\EE\biggl(\sum_{X\in\Pi'\cap(-\oo,-1)} p_X\biggr)\right).
\end{equation}
By  the Campbell--Hardy Theorem,
\begin{align}\label{g103}
\EE\biggl(\sum_{X\in\Pi'\cap(-\oo,-1)} p_X\biggr) &= 
\la \int_{-\oo}^{-1} dx \int_{0}^\oo dt\, \alpha e^{-\alpha t} \PP(F_{-x}\le t).
\end{align}
By \eqref{g101b}, we obtain after interchanging the order of integration that
\begin{align}\label{g103b}
\EE\biggl(\sum_{X\in\Pi'\cap(-\oo,-1)} p_X\biggr) &\le
%\la\int_{0}^\oo \alpha e^{-\alpha t} \EE(M_t)\,dt = 
\la \EE(M_{T_\alpha}).
\end{align}
By \eqref{g102}--\eqref{g103b}, 
\begin{equation}\label{g104}
\PP(A) \ge \exp\bigl(-\la -c\la\EE(M_{T_\alpha})\bigr)>0.
\end{equation}

For $k\in \ZZ$, let $B_k$ be the event that, for all $X\in\Pi'$, the diffusion $X+\zeta_X$
hits the interval $[k-1,k+1]$ only after time $T_X$. We repeat the argument of part (a) 
with \eqref{g104} in place of \eqref{eq:new506}, and thereby obtain the claim.
\end{proof}

\subsection{A condition for subcriticality when $\rho<\oo$}\label{sec:3-2}

Consider the general delayed diffusion model of Section \ref{sec:3-1}, and
assume first that $\rho\in(0,\oo)$.
Let $I_0=\{0\}$.
We call $y\in\Pi$ a \emph{first generation infected point up to time $t$} if $y$ is directly
infected by $P_0$ at or before time $t$. Let $I_{1,t}$ be the set of all first generation
infected points up to time $t$.
For 
$n\geq 2$, we call $z\in\Pi$ an \emph{$n$th generation infected point up to time $t$}
if, at or before time $t$,  $z$ is directly infected by some 
$y\in I_{n-1,t}$, and we define
$I_{n,t}$ accordingly. 
Write $I_n=\lim_{t\to\oo}I_{n,t}$,
the set of all $n$th generation infected points, and let
$I=\bigcup_n I_n$ be the set of points
that are ever infected.

In the following, we shall sometimes use the coupling of the delayed diffusion model with
the percolation-type system of the Section \ref{sec:prel}, and we shall use the notation of that section. 
In particular, we have that $I_1\subseteq J_0$, and by Proposition \ref{lem:3} that $I=\vec I$
(note that $I_1$ is a strict subset of $J_0$ if there exist $i,j\in J_0$ such that,
in the notation of \eqref{eq:defJ}, we have $0\to i\to j$).

\begin{proposition}\label{prop:1-1}
Consider the general delayed diffusion model.
Let $\rho\in(0,\oo)$ and
\begin{equation}\label{g0-1}
L_t(x)=
\EE\left(1-\exp\left(-\int_0^t \rho \mu(x-\zeta(s))\,ds\right)\right).
\end{equation}
We have that $\EE|I_{1,t}| \le R_t$ and $\EE|I_1| \le R$, where
\begin{align}\label{g1-1}
R_t &= \la\int_{\RR^d} \left[\int_0^t L_s(x)\alpha e^{-\alpha s} \, ds +L_t(x)e^{-\alpha t}\right]\, dx,  \\
R &= \lim_{t\to\oo} R_t =
\la \int_{\RR^d} \int_0^\oo L_s(x)\alpha e^{-\alpha s} \, ds\, dx.
\label{g2-1}
\end{align}
\end{proposition}

The constant $R$ in \eqref{g2-1}
is an upper bound for the so-called \emph{reproductive rate} of the process.
In the notation of Section \ref{sec:prel}, we have $R=\EE|J_0|$.

\begin{proposition}\label{prop:2-1}\mbox{\hfil}
Consider the general delayed diffusion model.
Let $\rho\in(0,\oo)$.
\begin{letlist}
\item We have that $\EE|I_n| \le R^n$ for $n \ge 0$, where $R$ is given in \eqref{g2-1}.

\item
If $R<1$, then $\EE|I| \le 1/(1-R)$, and hence $\thdd(\la,\rho,\alpha)=0$.
\item We have that $R\le \la\rho\,\iNT(\mu)/\alpha$.
\end{letlist}
\end{proposition}

Note that parts (b) and (c) imply that 
\begin{equation}\label{g91}
\thdd(\la,\rho,\alpha)=0 \q\text{if}\q \alpha>\la\rho\,\iNT(\mu).
\end{equation}

\begin{proof}[Proof of Proposition \ref{prop:1-1}]
Let $\sF_0(t)$ be the $\sigma$-field generated by
$(\zeta_0(s): 0\le s\le t)$.
Conditional on $\sF_0(t)$, for $i\ge 1$, let $A_i=(A_i^k: k \ge 0)$
be a Poisson process on $[0,\oo)$ with rate function 
$$
r_{X_i}(s):=\rho\mu(X_i-\zeta_0(s)), \qq s\in[0,\oo).
$$
Assume the $A_i$ are independent conditional on $\sF_0(t)$,
and write $N_i=|\{k: A_i^k \le t\}|$. 
We say that $P_0$ \lq contacts' $P_i$ at the times $\{A_i^k : k \ge 1\}$.
Let $U_t =\{X_i: i\ge 1,\ N_i \ge 1\}$ be the set of points in $\Pi$ that 
$P_0$ contacts up to time $t$. Note that $I_{1,t}$ is dominated stochastically
by $U_t$.  The domination is strict since there may exist $X_i\in U_t$ such that $P_i$ is 
infected before time $t$ by some previously infected  $P_j\ne P_0$. 

Consider a particle, labelled $P_j$ say,  with initial position $x\in\RR^d$.
Conditional on $\sF_0(t)$,
$P_0$ contacts $P_j$ prior to time $t$ with probability
$$
1-\exp\left(-\int_0^t r_x(s) \,ds\right).
$$
Therefore,
\begin{equation}
\label{g3-1}
\PP\bigl(X_j\in I_{1,t}\bigmid X_j=x,\, \sF_0(t)\bigr) \le \EE\left(
1-\exp\left(-\int_0^t r_x(s) \,ds\right)\biggmid\sF_0(t)\right).
\end{equation}

By the colouring theorem for Poisson processes 
(see, for example, \cite[Thm 6.13.14]{GS20}),
conditional on $\sF_0(t)$, $U_t$ is a Poisson process with inhomogeneous intensity function given by
$$
\Lambda_{t,\zeta_0}(x) = \la \EE\left(
1-\exp\left(-\int_0^t r_x(s) \,ds\right)\biggmid\sF_0(t)\right).
$$
By Fubini's theorem,
\begin{align}\label{g4-1}
\EE|I_{1,t}| &\le \EE\bigl(\EE(|U_t|\bigmid T_0)\bigr)\\
&=\int_{\RR^d} \left[\la \int_0^t L_s(x)\alpha e^{-\alpha s} \, ds 
+L_t(x)\PP(T_0>t)\right]\, dx, 
\nonumber
\end{align}
and \eqref{g1-1} follows. Equation \eqref{g2-1} follows as $t\to\oo$ 
by the monotone and bounded convergence theorems.
\end{proof}

\begin{proof}[Proof of Proposition \ref{prop:2-1}]
(a) 
This may be proved directly, but it is more informative to use the percolation representation of Section \ref{sec:prel}.
Let $\vec G=\vec G_\Pi$ be as defined there, and note that, in the given coupling, we have
$I_n=\{i\in\ZZ_0: \de(0,i) = n\}$
where $\de$ denotes graph-theoretic distance on $\vec G$. 

We write
\begin{equation*}
|I_n| \le \sum_{i\in\ZZ_0} |J_i|1(i\in I_{n-1}).
\end{equation*}
By the independence of $J_i$ and the event $\{i\in I_{n-1}\}$,
\begin{equation*}
\EE|I_n| \le \sum_{i\in\ZZ_0} \EE|J_i|\, \PP(i\in I_{n-1})
\le R\EE|I_{n-1}|,
\end{equation*} 
and the claim follows.

(b) By part (a) and the assumption $R<1$,
$$
\EE|I|=\sum_{n=0}^\oo \EE|I_n| \le \frac 1{1-R} < \oo.
$$
Therefore, $\thdd(\la,\rho,\alpha) =\PP(|I|=\oo)=0$.

(c) Since $1-e^{-z}\le z$ for $z\ge 0$, by \eqref{g0-1} and Fubini's theorem, 
$$
\int_{\RR^d} L_t(x)\, dx \le \rho t\,\iNT(\mu).
$$
By \eqref{g2-1},
$$
R \le \la\rho\,\iNT(\mu)\int_0^\oo  s\alpha e^{-\alpha s}\,ds = \frac{\la\rho}\alpha\,\iNT(\mu),
$$
as claimed.
\end{proof}

\subsection{Infection with compact support}\label{sec:3-3}
Suppose $\mu=1_M$ with $M$ compact. By \eqref{g0-1} and \eqref{g2-1},
$R=R(\rho)$ is given by 
\begin{equation}\label{6-1}
R(\rho) = \la \int_{\RR^d} \int_0^\oo L_s(x)\alpha e^{-\alpha s} \, ds\, dx,
\end{equation}
where
\begin{equation}\label{g7-1}
L_t(x)=
\EE\bigl(1-\exp\left(-\rho Q_t(x)\right)\bigr), 
\end{equation}
and
$$
Q_t(x) = \bigl|\{s\in[0,t] : x\in\zeta(s) + M\}\bigr|_1.
$$
We denote by $\Si_t$ the $M$-sausage of $\zeta$, that is,
\begin{equation}\label{sausage}
\Si_t:= \bigcup_{s\in [0,t]} \bigl[\zeta(s) + M\bigr], \qq t \ge 0.
\end{equation}

Consider the limit $\rho\to\oo$.
By \eqref{6-1} and dominated convergence,
\begin{equation}\label{g7-1c}
R(\rho) \uparrow \ol R:=  
\la \int_{\RR^d} \int_0^\oo \ol L_s(x)\alpha e^{-\alpha s} \, ds\, dx,
\end{equation}
where
$$
\ol L_t(x)=\PP(Q_t(x)>0) = \PP(x \in \Si_t).
$$
Therefore,
\begin{equation}\label{g7-0}
\ol R = \la \int_0^\oo \EE|\Si_s|_d\,  \alpha e^{-\alpha s}\,ds,
\end{equation}
where the integral is the mean volume of the sausage $\Si$ up to time $T_0$.
This formula is easily obtained  from first principles
applied to the $\rho=\oo$ delayed
diffusion process (see Section \ref{sec:3-3z}).

\begin{example}[Bounded motion]\label{ex1}
If, in addition to the assumptions above, each particle is confined within some given
distance $\De<\oo$ of its initial location, then
$\Si_t\subseteq S(\De+\rad(M))$.
Therefore, by \eqref{g7-1c}--\eqref{g7-0},
\begin{equation}\label{g-45}
R(\rho)\le \ol R  \leq \la \bigl|S(\De+\rad(M))\bigr|_d.
\end{equation}
If the right side of \eqref{g-45} is strictly less than $1$,  
then $\thdd(\la,\rho,\alpha)=0$ for $\rho\in(0,\oo)$ by Proposition \ref{prop:2-1}.
This is an improvement over \eqref{g91} for large $\rho$.
\end{example}

\subsection{A condition for subcriticality when $\rho=\oo$}\label{sec:3-3z}
Let $d \ge 2$, $\rho=\oo$, and $\mu=1_M$ with $M$ compact.
The argument of Sections \ref{sec:3-2}--\ref{sec:3-3} is easily adapted subject to a condition on
the volume of the sausage $\Si$ of \eqref{sausage},
namely
\begin{equation}\label{saus-cond}
\text{\Cs:  for $t \ge 0$, $\EE|\Si_t|_d \le \ga e^{\si t}$,}
\end{equation}
for some $\ga,\si\in[0,\oo)$.
Let 
\begin{equation}\label{g88}
R(\oo) = \la \int_0^\oo \EE|\Si_s|_d\,  \alpha e^{-\alpha s}\,ds,
\end{equation}
in agreement with \eqref{g7-1c}--\eqref{g7-0}.
Note that $R(\oo)$ 
equals the mean number of points of the Poisson process $\Pi\sm\{0\}$
lying in the sausage $\Si_T$, where $T$ is independent of $\Si$ and is
exponentially distributed with parameter $\alpha$.

\begin{theorem}\label{thm:4}
\mbox{\hfil}
\begin{letlist}

\item  If $R(\oo)<1$ then $\thdd(\la,\oo,\alpha)=0$.

\item
Assume condition \Cs\ of 
\eqref{saus-cond} holds, and $\la<\ula:=1/\ga$. If $\alpha>\oac:=\si/(1-\la\ga)$,
then $R(\oo)<1$ for $\alpha>\oac$. 

\end{letlist}
\end{theorem}

\begin{proof}
(a) This holds by the argument of  Proposition
\ref{prop:2-1} adapted to the case $\rho=\oo$. 

(b) Subject to condition \eqref{saus-cond} with $\la\ga<1$,
\begin{equation}\label{g89}
R(\oo) \le \la\int_0^\oo \alpha\ga e^{-(\alpha-\si)s}\,ds =\frac{\la\alpha\g}{\alpha-\si},\qq
\alpha>\si,
\end{equation}
and the second claim follows. 
\end{proof}

\begin{example}[Brownian motion with $d=2$]\label{ex:bm}
Suppose $d=2$, $\zeta$ is a standard Brownian motion, and $M=S$.
By \eqref{g88}  
and the results of Spitzer \cite[p.\ 117]{FP64},
\begin{align*}
R(\oo)
&=  \la |S|_2+\la \int_0^{\infty}\alpha e^{-\alpha s}\int_{\RR^2\sm S}
\PP (x\in \Si_s)\,dx\,ds\\
&= \la\pi+\la \int_{\RR^2\sm S} \frac{K_0(\|x\|_2\sqrt{2\alpha} )}{K_0(\sqrt{2\alpha)}}\,dx
= \la Z_\alpha,
\end{align*}
where
\begin{equation}\label{g14}
Z_\alpha=\pi+\frac{2\pi}{\sqrt \alpha} \frac{K_1(\sqrt{2\alpha})}{K_0(\sqrt{2\alpha})}
= \pi + \frac{2\pi}{\sqrt \alpha} + \o(\alpha^{-\frac12})\qq\text{as } \alpha\to\oo.
\end{equation}
Here,  $K_1$ (\resp, $K_0$) is the modified Bessel function of the second kind of order $1$ (\resp, order $0$) given by
\begin{equation*}
K_0(x)=\int_0^{\infty} e^{-x\cosh s}\,ds,\qq
K_1(x)=\int_0^{\infty} e^{-x\cosh s}\cosh s\, ds.
\end{equation*}
Therefore, if $\la<\ula:=1/ \pi$, there exists $\oac\in(0,\oo)$ such that $R(\oo)<1$ when $\alpha>\oac$.
\end{example}

\begin{example}[Brownian motion with $d \ge 5$]\label{ex:bm2}
Suppose $d\ge 5$, $\zeta$ is a standard Brownian motion, and $M=S$.
Getoor \cite[Thm 2]{Get} has shown an explicit constant $C$ such that
$$
\EE|\Si_t|_d -tc_d \uparrow C\qq\text{as } t \to\oo,
$$
where $c_d$ is the Newtonian capacity of the closed unit ball $S$ of $\RR^d$.
By \eqref{g88}, 
$$
R(\oo) \le\la\left( \frac{c_d}\alpha+C\right).
$$
Therefore, if $\la<\ula:=1/C$, there exists $\oac\in(0,\oo)$ such that $R(\oo)<1$ when $\alpha>\oac$.
Related estimates are in principle valid for $d=3,4$, though the 
behaviour of $\EE|\Si_t|_d-tc_d$ is more complicated
(see \cite{Get}).
\end{example}

\begin{example}[Brownian motion with constant drift]\label{ex:bmd}
Let $d \ge 2$, $M=S$, with $\zeta$ a Brownian motion with constant drift.
It is standard (with a simple proof using subadditivity) that the limit
$\ga:=\EE|\Si_t|_d/t$ exists and in addition is strictly positive when the drift is non-zero.
Thus, for $\eps>0$, there exists $C_\eps$ such that
$$
\EE|\Si_t|_d \le C_\eps+(1+\eps)\ga t, \qq t \ge 0.
$$
As in Example \ref{ex:bm2}, if $\la<\ula:=1/C_\eps$, 
there exists $\oac\in(0,\oo)$ such that $R(\oo)<1$ when $\alpha>\oac$.
See also \cite{HM17, HHNZ10}.
\end{example}

\begin{example}[Ornstein--Uhlenbeck process]\label{ex:OU}
Let $M=S$ and consider the  Ornstein--Uhlenbeck process on $\RR^d$ satisfying
\begin{equation*}
d \zeta(t)=A \zeta(t)\, dt +d W_t
\end{equation*}
where $W$ is standard Brownian motion, $A$ 
is a $d\times d$ real matrix, and $\zeta(0)=0$. The solution to this stochastic differential equation is
$$
\zeta(t) = \int_0^t e^{A(t-s)}\,dW_s,
$$
so that $\|\zeta(t)\|_d \le e^{|A|t}\|X_t\|_d$ where
$$
X_t=\int_0^t e^{-A s} \, dW_s
$$
defines a martingale, with $|A|$ denoting operator norm. 
By the Burkholder--Davis--Gundy inequality applied to $X$ (see, for example, \cite[Thm 1.1]{MaRo}),
the function $M_t=\sup\{\|\zeta(s)\|_d: s\in[0,t]\}$ satisfies
$$
\EE(M_t^d) \le c  e^{d|A| t}\left(\int_0^t e^{2|A|s} \,ds\right)^{d/2},
$$ 
for some $c<\oo$. Now,
$$
|\Si_t|_d \le 2^d(1+M_t)^d,
$$
whence Condition \Cs\ holds for suitable $\ga,\si<\oo$.
\end{example}

\subsection{Proof of Theorem \ref{thm:1x}}\label{sec:3}

This is proved in several stages, as described in the next subsections. 
\begin{numlist}
\item[\S\ref{sec3.6.1}]  The existence and some basic properties of $\ac(\lambda,\rho)$
are proved. 
\item[\S\ref{ssec:d=2}] Let $d=2$.
Suppose  $\mu=1_S$. The remaining properties of  $\ac$ are established in the respective cases
$\rho=\oo$ and $\rho\in(0,\oo)$.
\item[\S\ref{ssec3.6.3}] The previous results are proved for general $\mu$ and $\rho\in(0,\oo)$.
\item[\S\ref{ssec:dge3}] Corresponding statements are proved for $d\ge 3$.
\end{numlist}

\subsubsection{Existence of $\ac$}\label{sec3.6.1}
Consider the Brownian delayed diffusion model with $d\ge 2$, $\rho\in(0,\oo]$.
When $\rho=\oo$, we assume in addition that 
\begin{equation}\label{g20+}
\mu(x)= 1_S(x) ,\qq x \in \RR^2,
\end{equation}
where $S$ is the closed unit ball with centre at the origin; note in this case that
$\mu$ is radially decreasing.

By Proposition \ref{prop:4}(a,\,d), $\thdd(\la,\rho,\alpha)$ is non-decreasing in $\rho$, and non-increasing
in $\alpha$, and is moreover non-decreasing in $\la$ 
if $\rho=\oo$ (the radial monotonicity of $\mu$ has been used in this case). With
$$
\ac(\la,\rho) := \inf\bigl\{\alpha: \thdd(\la,\rho,\alpha)=0\bigr\},
$$
we have that
\begin{equation*}
\thdd(\la,\rho,\alpha) 
\begin{cases} >0 &\text{if } \alpha<\ac(\la,\rho),\\
=0 &\text{if } \alpha>\ac(\la,\rho),
\end{cases}
\end{equation*}
and, furthermore, $\ac$ is non-decreasing in $\rho$.

In case (a) of the theorem, by Proposition \ref{prop:2-1}, $\ac(\la,\rho)<\oo$
for all $\la$, $\rho$. In case (b), by Theorem \ref{thm:4} and Example \ref{ex:bmd},
there exists $\ula\in(0,\lac]$ such that  $\ac(\la,\oo)<\oo$ when $\la\in(0,\ula)$.
As remarked after \eqref{g3}, $\ac(\la,\oo)=\oo$ when
$\la>\lac$.

It remains to show that $\ac(\la,\rho)>0$ for all $\la\in(0,\oo)$, $\rho\in(0,\oo]$, 
and the rest of this proof is devoted to that.
This will be achieved by comparison
with a directed site percolation model on $\ZZ_0^2$ viewed as a directed graph 
with edges directed away from the origin.
When $d=2$, the key fact is the 
\emph{recurrence} of Brownian motion, which permits a static block argument. This fails when $d \ge 3$,
in which case we employ a dynamic block argument and the 
\emph{transience} of Brownian motion.

\subsubsection{The case $d=2$ with $\mu=1_S$}\label{ssec:d=2}

Assume first that $d=2$, for which we use a static block argument.
Let $\eps>0$. We choose $a>0$ such that
\begin{equation}\label{g30}
\PP(\Pi' \cap aS\ne \es)>1-\eps,
\end{equation}
where $\Pi'=\Pi\sm\{0\}$.
For $\bx\in\ZZ^2$, let $S_\bx=3a\bx +aS$ be the ball with 
radius $a$ and centre at $3a\bx$. We declare $\bx$  \emph{occupied}
if $\Pi\cap S_\bx \ne\es$, and \emph{vacant} otherwise; thus, the origin $0$ is
invariably occupied. 
Note that the occupied/vacant states of different $\bx$ are independent.
If a given $\bx\ne 0$ is occupied, we let $Q_\bx\in \Pi\cap S_\bx$ be the
least such point in the lexicographic ordering, and we set $Q_0=0$.
If $\bx$ is occupied, we denote by $\zeta_{\bx}$ the diffusion associated with the particle at $Q_\bx$,
and $T_\bx$ for the lifetime of this particle.

Let $\zeta$ be a standard Brownian motion on $\RR^2$ with $\zeta(0)=0$, and let
\begin{equation}\label{g92}
\ws_t(\zeta) := \bigcup_{s\in[0,t]}\bigl[\zeta(s)+S\bigr], \qq t\in[0,\oo),
\end{equation}
be the corresponding Wiener sausage.

\emph{Suppose for now that $\rho=\oo$}; 
later we explain how to handle the case $\rho<\oo$.
First we explain what it means  to say that the origin $0$ is \emph{open}.
Let
$$
F(\zeta,z)=\inf\{t: z\in \ws_t(\zeta)\},\qq z \in\RR^2,
$$
be the first hitting time of $z$ by $\ws(\zeta)$.

For $\by\in\ZZ^2$, we define the event 
$$
K(\zeta_0,\by)=\bigcap_{z\in S_\by} \{F(\zeta_0,z)<T_0\},
$$
and
$$
K(\zeta_0)=\bigcap_{\by\in N} K(\zeta_0,\by),
$$
where $N=\{(0,1), (1,0)\}$ is the neighbour set of $0$ in the directed graph on 
$\ZZ_0^2$.
By the recurrence of $\zeta_0$,
we may choose $\alpha>0$ sufficiently small  that
\begin{equation}\label{g31}
p_\alpha(0):=\PP(K(\zeta_0)) 
\q\text{satisfies} \q p_\alpha(0)>1-\eps.
\end{equation}
We call $0$ \emph{open}  if the event  $K(\zeta_0)$ occurs. 
If $0$ is not open, it is called \emph{closed}. (Recall that $0$ is automatically occupied.)

We now explain what is meant by declaring $\bx\in\ZZ^2\sm\{0\}$ to be open.
Assume $\bx$ is occupied and pick $Q_\bx$ as above.
For $\by\in\bx+N$, we define the event 
\begin{equation}\label{eq:g36}
K(\zeta_\bx,\by)=\bigcap_{z\in S_\by} \{F(Q_\bx+\zeta_\bx,z)<T_\bx\},
\end{equation}
and
$$
K(\zeta_\bx)=\bigcap_{\by\in N} K(\zeta_\bx,\by).
$$
By the recurrence of $\zeta$,
we may choose $\alpha$ such that
\begin{equation}\label{g31b}
p_\alpha(\bx):=\PP\bigl(K(\zeta_\bx)\bigmid\text{$\bx$ is occupied}\bigr) 
\q\text{satisfies} \q p_\alpha(\bx)>1-\eps.
\end{equation}
We declare $\bx\in\ZZ^2$
\emph{open} if $\bx$ is occupied,
and in addition the event 
$K(\zeta_\bx)$ occurs. 
A vertex of $\ZZ^2$ which is not open is called \emph{closed}.
Conditional on the set of occupied vertices, the open/closed states are independent.

The open/closed state of a vertex $\bx\in\ZZ^2$ depends only on the 
existence of $Q_\bx$ and on the diffusion $\zeta_\bx$, whence the open/closed states of
different $\bx\in\ZZ^2$ are independent.
By \eqref{g30}--\eqref{g31}, the configuration of 
open/closed vertices forms a family of independent Bernoulli random variables with
density at least $(1-\eps)^2$. Choose $\eps>0$ such that
$(1-\eps)^2$ exceeds the 
critical probability of directed site percolation on $\ZZ_0^2$ (cf.\ 
\cite[Thm 3.30]{G-pgs}). With strictly positive probability, 
the origin is the root of an infinite directed cluster of the latter process.
Using the definition of the state \lq open' for the delayed diffusion  model,
we conclude that the graph $\vec G$ (of Section \ref{sec:prel}) contains an infinite directed path 
from the origin with strictly positive probability.
The corresponding claim of Theorem \ref{thm:1x}(b) follows by Lemma \ref{lem:3}.

\emph{Suppose now that $\rho\in(0,\oo)$}. We adapt the above argument by 
redefining the times $F(\zeta,z)$ and the events $K(\zeta)$ as follows.
Consider first the case of the origin. Let
\begin{equation}\label{g39}
E(\zeta, z, t)=\bigl|\{s\in[0,t]: z\in \zeta(s)+S\}\bigr|_1, \qq z\in\RR^2.
\end{equation}
Pick $F>0$ such that $e^{-\rho F}< \eps$, and write
$$
\ol K(\zeta_0,t) = \bigcap_{\by\in N,\,  z\in S_\by} \{E(\zeta_0,z,t)>F\}.
$$
In words, $\ol K(\zeta_0,t)$ is the event that
the Wiener sausage, started at $0$ and run for time $t$,
contains every $z \in S_{(0,1)} \cup S_{(1,0)}$ for an aggregate  time
exceeding $F$.  It follows that, given that $Q_\by\in \Pi\cap S_\by$ for some $\by\in N$,
then $P_0$ infects $Q_\by$ with probability at least $1-e^{-\rho F}>1-\eps$.  

By elementary properties of a recurrent Brownian motion,
we may pick $t$ and then $\alpha=\alpha(t)$ such that 
(cf.\ \eqref{g31})
\begin{equation}\label{g36}
p_\alpha(0):=\PP\bigl(\ol K(\zeta_0,t)\cap\{t<T_0\}\bigr) 
\q\text{satisfies} \q p_\alpha(0)>1-\eps.
\end{equation}

Turning to general $\bx\in\ZZ^2\sm\{0\}$, 
a similar construction is valid for an event $\ol K(\zeta_\bx ,t)$ as in \eqref{g36},
and we replicate the above comparison with directed percolation  
with $(1-\eps)^2$ replaced by $(1-\eps)^3$. 

\subsubsection{The case $d=2$ with general $\mu$ and $\rho<\oo$}\label{ssec3.6.3}
We consider next the Brownian delayed diffusion process in two dimensions
with infections governed by the pair
$(\rho,\mu)$, as described in Section \ref{sec:3-1}. Assume that $\rho\in(0,\oo)$ and $\iNT(\mu)\in(0,\oo)$.
The basic method is to adapt the arguments of Section \ref{ssec:d=2}. The new ingredient is a proof of
a statement corresponding to \eqref{g36}, as follows. 

Let $\by\in N$ and write $S_\by=3a\by+aS$ as before. For $\eps>0$, pick $a$ such that
$\PP(\Pi\cap S_\by\ne\es)> 1-\eps$. 
Suppose that $\Pi\cap S_\by\ne\es$,
and write $Q:=Q_\by$ for the least point in the lexicographic ordering of $\Pi\cap S_\by$. 
Consider $Q$  henceforth as given. 
The following concerns only two particles, namely $P_0$ and the particle $P$ at $Q$.
Consider the process in which $P_0$ diffuses forever according to $\zeta:=\zeta_0$, and $P$ remains stationary.
Given $\zeta$ and $Q$, let $A$ be a Poisson process of times $(A_k: k=1,2,\dots)$ with rate function
$r(s) = \rho \mu(Q-\zeta(s))$. We say that $P_0$ \lq contacts' $P$ at the times of $A$, and we claim that
\begin{equation}\label{eq:new124}
\PP(A_1<\oo)=1.
\end{equation}
This implies that, for $\eps>0$ there exists $t$ such that $\PP(A_1<t)>1-\eps$, and 
we may then pick $\alpha>0$ sufficiently small that  $\PP(A_1<T_0) > 1-\eps$, where $T_0$ is the lifetime of $P_0$.
Therefore, subject to \eqref{eq:new124}, $P_0$ infects $P$ with probability at least $1-2\eps$.
This is enough to allow the argument of Section \ref{ssec:d=2} to proceed, and we turn to the proof of \eqref{eq:new124}.

Fix $\bz\in\RR^2$ to be chosen soon, and write $T_b$ for the disk $Q-\bz +bS$.
By the Lebesgue density theorem (see, for example, \cite[Cor.\ 2.14]{Matt}), 
we may pick $\bz\in\RR^2$ and $\eta>0$ such that
\begin{equation}\label{eq:new125}
\int_{T_2} \mu(Q-\bu)\, d\bu 
= \int_{\bz+2S} \mu(\bv)\, d\bv
\ge \eta \mu(\bz)>0.
\end{equation}

We shall suppose without loss of generality that $0\notin T_1$.
Let $H$ be the hitting time (by $\zeta$) of the disk $T_1$ and let $H'>H$
be the subsequent exit time of the disk $T_3$. The probability that $P_0$ contacts $P$ during the 
time-interval $(H,H')$ is
\begin{equation}\label{eq:new127} 
p := 1- \EE \left[\exp\left( -\rho\int_{H}^{H'} \mu(Q-\zeta(t))\,dt\right)\right].
\end{equation}
By spherical symmetry and \cite[Thm 3.31]{MorP},
$$
\EE\int_{H}^{H'} \mu(Q-\zeta(t))\,dt = \int_{T_3} \mu(Q-\bu)G(\bx,\bu)\,d\bu
$$
for any given fixed $\bx \in \pd T_1$, where $G$ is the 
appropriate Green's function of \cite[Lem.\ 3.36]{MorP}.
There exists $c>0$ such that $G(\bx,\bu)\ge c$ for $\bu\in T_2$, so that
\begin{align*}
\EE\int_{H}^{H'} \mu(Q-\zeta(t))\,dt \ge \int_{T_2} \mu(Q-\bu)G(\bx,\bu)\,d\bu
\ge c\eta  \mu(\bz)   >0,
\end{align*}
by \eqref{eq:new125}.
By \eqref{eq:new127}, we have $p>0$.

We now iterate the above. Each time $\zeta$ revisits $T_1$, 
having earlier departed from $T_3$,
there is probability $p$ of such a contact. These contact events are independent, 
and, by recurrence, a.s.\ some such contact occurs ultimately.
Equation \eqref{eq:new124} is proved.

\subsubsection{The case $d \ge 3$.}\label{ssec:dge3}
Let $d=3$; the case $d \ge 4$ is handled similarly. 
This time we use a \emph{dynamic} block argument, combined with Remark \ref{rem:2}.
The idea is the following.
Let $\zeta_0$ be the diffusion of particle $P_0$.
We track the projection of $\zeta_0$, denoted $\ol\zeta_0$, on the plane $\RR^2\times\{0\}$.
By the recurrence of $\ol\zeta_0$, the Wiener sausage $\ws(\zeta_0)$ 
a.s.\ visits
every line $\bz\times\RR$ infinitely often, 
for $\bz \in \RR^2$ (such $\bz$ will be chosen later).  At such a visit,
we may choose a point $Q_\bz'$ of $\Pi$ lying in $\ws(\zeta_0)$ \lq near to' 
the line $\bz\times\RR$.
The construction is then iterated with $Q_\bz'$ as the starting particle. 
We build this process in each of two independent directions, 
and may choose the parameter
values such that it dominates the cluster at $0$ of a supercritical directed site percolation process.

For $A \subseteq \RR^3$, we write
$\ol A$ for its projection onto the first two coordinates.
We abuse notation by  identifying $\bx=(x_1,x_2,0,\dots,0)\in\ol{\RR^3}$ 
(\resp, $\ol{\ZZ^3}$, etc) with 
the $2$-vector $\bx=(x_1,x_2) \in \RR^2$ (\resp, $\ZZ^2$, etc).
Thus, $\ol{\RR^3}=\RR^2\times\{0\}$ is the plane of the first two coordinates,
and similarly $\ol{\ZZ^3}=\ZZ^2\times\{0\}$, $\ol{\ZZ_0^3}=\ZZ_0^2\times\{0\}$,
and $\ol S=S \cap \ol{\RR^3}$. 

For  
$\bx\in \ol{\ZZ^3}$, let $\ol S_\bx=3a\bx +a\ol S$ be the two-dimensional ball with 
radius $a>1$ and centre at $3a\bx$, and let $C_\bx=\ol S_\bx\times\RR$ be the \emph{cylinder}
generated by $\bx$. We explain later how $a$ is chosen.
Let $\zeta=(\zeta^{(i)}:i=1,2,3)$ be 
a standard Brownian motion in $\RR^3$ with 
$\zeta(0)=0$ and coordinate processes $\zeta^{(i)}$, 
and let $\ol\zeta=(\zeta^{(1)},\zeta^{(2)},0)$ 
be its projection onto the
first two coordinates. Note that $\ol\zeta$ is a recurrent process on $\ol{\RR^3}$.

We declare the particle at $0$ to be \emph{open}, and let $\by \in N:=\{(1,0), (0,1)\}$.
We shall see that, with a probability to be bounded below, 
there exists a (random) particle at some $Q_\by\in C_\by$ 
such that $P_0$ infects this particle.  If this occurs, we declare $\by$ to be open.
On the event that $\by$ is open, we may iterate the construction starting at 
$Q_\by$, to find a number of further random vertices of $\vec G$.
By a comparison with a supercritical directed site percolation model,
we shall show (for large $\alpha$) that $\vec G$ contains an infinite directed cluster
with root $0$. The claim then follows by Proposition \ref{lem:3} and Remark \ref{rem:2}.

\begin{figure}
\includegraphics[width=0.5\textwidth]{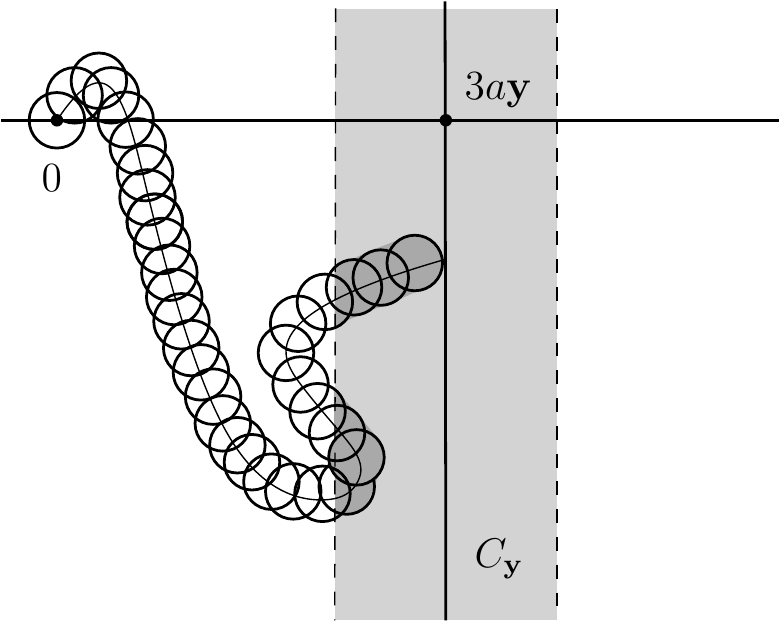}
\caption{The Wiener sausage $\ws(\zeta_0)$ stopped when it hits the line $(3a\by)\times\RR$.
The dark shaded areas constitute the region $L(\zeta_0,\by)$.}
\label{fig:2}
\end{figure}

\emph{Suppose for now that $\rho=\oo$.}
Let $\eps>0$. 
With $\zeta$ a standard Brownian motion on $\RR^3$
with $\zeta(0)=0$, let $\ws_t(\zeta)$
be the corresponding Wiener sausage \eqref{g92}.
We explain next the state open/closed for a vertex $\by\in N$.
Let
\begin{equation}\label{g37-}
F(\zeta_0,\by) = \inf\big\{t: ((3a\by)\times\RR)\cap  \ws_t(\zeta_0)\ne\es\bigr\}.
\end{equation}
Since 
$\ol\zeta_0$ is recurrent,  we have $F(\zeta_0,\by)<\oo$ a.s.
Let $T_0$ be the lifetime of $P_0$, and define the event
\begin{equation}\label{g37}
K(\zeta_0,\by)=  \bigl\{F(\zeta_0,\by)<T_0\bigr\}.
\end{equation}

We explain next how $a$ is chosen (see Figure \ref{fig:2}). Let $a>1$ and, for $\by\in N$,
consider the intersection 
$$
L(\zeta_0,\by):= \ws_{F(\zeta_0,\by)}(\zeta_0) \cap C_\by.
$$

\begin{lemma}\label{lem:geom}
There exists $c>0$ such that the volume of $L(\zeta_0,\by)$ satisfies
$$
|L(\zeta_0,\by)|_3 \ge ca.
$$
\end{lemma}

\begin{proof}
The set $L(\zeta_0,\by)$ is the union of disjoint subsets of the Wiener sausage, exactly one of which, denoted $L'$,
touches the line $(3a\by)\times\RR$. The volume of $L'$ is bounded below by the volume of
the union of a cylinder with radius $1$ and length $a-1$, and a half-sphere with radius $1$. Thus,
$$
|L(\zeta_0,\by)|_3 \ge (a-1)\pi+ \tfrac23 \pi\ge\tfrac23\pi a,
$$
whence the lemma holds with $c=\frac23\pi$.
\end{proof}

By Lemma \ref{lem:geom}, we may pick $a >1$ sufficiently large that 
$$
\PP\bigl(N_\by \mid K(\zeta_0,\by)\bigr) > 1-\eps \qq 
\text{where}\qq N_\by:=\{ \Pi\cap L(\zeta_0,\by)\ne \es\}.
$$
If $\Pi\cap L(\zeta_0,\by)\ne \es$, we pick the least point in the intersection 
(in lexicographic order) and denote it $Q_\by$, and we say that $Q_\by$ 
has been \emph{occupied from $0$}. We call $\by$ \emph{open}
if $K(\zeta_0,\by)\cap N_\by$ occurs, and \emph{closed} otherwise.

By the recurrence of $\ol\zeta$,
we may choose $\alpha>0$ such that, for $\by\in N$,
\begin{equation}\label{g38}
p_\alpha(\by):=\PP(\text{$\by$ is open}) 
\q\text{satisfies} \q p_\alpha(\by)>1-\eps.
\end{equation}

In order to define the open/closed states of other $\bx\in\ol{\ZZ^3}$, it is necessary
to generalize the above slightly, and we do this next. Instead of considering a Brownian motion $\zeta$
starting at $\zeta(0)=0$, we move the starting point to some $q\in \ol{\RR^3}$.  
Thus $\zeta$ becomes $q+\zeta$, and \eqref{g37-}--\eqref{g37} become
\begin{align*}
F(\zeta,q,\by) &= \inf\big\{t: ((3a\by) \times\RR)\cap (q+\ws_t(\zeta))\ne \es\bigr\},\\
K(\zeta,q,\by,T) &= \{F(\zeta,q,\by)<T\}.
\end{align*}
By the recurrence of $\ol\zeta$,
we may choose $\alpha$ such that
\begin{equation}\label{g38+}
\ol p_\alpha(\by):=\inf\bigl\{\PP(K(\zeta_0,q,\by,T_0)): q\in \ol S\bigr\} 
\q\text{satisfies} \q \ol p_\alpha(\by)>1-\eps.
\end{equation}
The extra notation introduced above will be used at the next stage.

We construct a non-decreasing sequence pair $(V_n,W_n)$ 
of disjoint subsets of $\ol{\ZZ_0^3}$ in the following way.
The set $V_n$ is the set of vertices 
known to be open at stage $n$ of the construction, and $W_n$ is the set known to be closed.
Our target is to show that the $V_n$ dominate some supercritical percolation process.
 
The vertices of $\ol{\ZZ_0^3}$ are ordered in
$L^1$ order: for $\bx=(x_1,x_2)$,   $\by=(y_1,y_2)$,
we declare
$$
\bx < \by \qq\text{if} \qq \text{either $x_1+x_2<y_1+y_2$,\q or
$x_1+x_2=y_1+y_2$ and $x_1<y_1$}.
$$
Let $G_n = \{(x_1,x_2)\in\ZZ_0^2: x_1+x_2=n\}$, and call $G_n$ the $n$th generation of $\ZZ_0^2$. 

First, let
\begin{equation*}
V_0= \{0\},\qq W_0=\es.
\end{equation*}
We choose the least $\by\in N$, and set:
\begin{align*}
\text{if $\by$ is open:}\q
&V_1= V_0\cup\{\by\},\, W_1=W_0,\\
\text{otherwise:}\q &V_1=V_0,\, W_1=W_0\cup \{\by\}. 
\end{align*}
In the first case, we say that \lq $\by$ is occupied from $0$'.

For $A \subset \ol{\ZZ_0^3}$, let $\De A$ be the set of vertices $b\in\ol{\ZZ_0^3}\sm A$ 
such that $b$ has some neighbour $a\in A$ with $a<b$. 
Suppose $(V_k,W_k)$ have been defined for $k=1,2,\dots,n$, and define 
$(V_{n+1},W_{n+1})$ as follows.
Select the least $\bz\in\De V_n\setminus W_n$. 
If such $\bz$ exists, find the least $\bx\in V_n$ 
such that $\bz=\bx + \by$ for some $\by\in N$.
Thus $\bx$ is known to be open, and there exists a vertex of $\vec G$ at the point
$Q_\bx \in C_\bx$. 

As above,
\begin{align*}
L(\zeta_\bx,Q_\bx,\bz)&:= \ws_{F(\zeta_\bx,Q_\bx,\by)}(Q_\bx + \zeta_\bx) \cap C_\bz,\\
N_\bz &:=\{ \Pi\cap L(\zeta_\bx,Q_\bx,\by)\ne \es\}.
\end{align*}
If $K(\zeta_\bx,Q_\bx,\bz,T_\bx)\cap N_\bz$ occurs
we call $\bz$ \emph{open}, and we say that $\bz$ is occupied from $\bx$;
otherwise we say that $\bz$ is \emph{closed}.
 \begin{align*}
\text{If $\bz$ is open:}\q
&V_{n+1}= V_n\cup\{\bz\},\, W_{n+1}=W_n,\\
\text{otherwise:}\q &V_{n+1}=V_n,\, W_{n+1}=W_n\cup\{\bz\}. 
\end{align*}
By \eqref{g38}--\eqref{g38+},
the  vertex $\bz$ under current scrutiny is open with conditional probability 
at least $(1-\eps)^2$. 

This process is iterated until the earliest stage at which no such $\bz$ exists.
If this occurs for some 
$n<\oo$, we declare $V_m=V_{n}$ for $m \ge n$, and in any case we set 
$V_\oo =\lim_{m\to\oo} V_m$.

The resulting set $V_\oo$ is the cluster at the origin of a type of 
dependent directed site  percolation process which is built by generation-number. 
Having discovered the open vertices $\bz$ in
generation $n$ together with the associated points $Q_\bz$,  the law 
of the next generation is (conditionally) independent of the past and is
$1$-dependent.

We now apply a stochastic-domination argument. Such methods have been used since at least \cite{DP96}, 
and the following core lemma was systematized by Liggett, Schonmann, and Stacey \cite[Thm 0.0]{LSS} 
(see also \cite[Thm 7.65]{G99} and the references therein). 
Let $\de\in(0,1)$, and let $\bX=(X_x: \bx\in\ZZ_0^2)$ be a $1$-dependent family of Bernoulli random variables
such that $\EE(X_\bx) > 1-\de$ for all $\bx$. 
There exists $\eta(\de)>0$, satisfying $\eta(\de)\to 0$ as $\de
\to\ 0$, such that $\bX$ dominates stochastically a family $\bY=(Y_x: \bx\in\ZZ_0^2)$
of independent Bernoulli variables with parameter $1-\eta(\de)$.
We choose $\de>0$ such that $1-\eta(\de)$ exceeds the critical probability of directed site percolation on $\ZZ_0^2$.
By the above, for sufficiently small $\de>0$,
there is strictly positive probability of an infinite directed path on $\ZZ_0^2$ comprising vertices $\bx$
with $X_\bx=1$.

With $\de$ chosen thus and $\eps=\de/2$, we deduce as required that $\PP(|V_\oo|=\oo)>0$.
By a consideration of the geometry of
the above construction, and the definition of the local states open/occupied,
by \eqref{g93} this entails $\thdd(\la,\oo,\alpha)>0$.

A minor extra complication arises at the last stage, in that the events $\{\bx\text{ is open}\}$ are not $1$-dependent,
but only $1$-dependent within a given generation conditional on earlier generations. 
This may be viewed as follows. Begin with a family $\bY=(Y_x: \bx\in\ZZ_0^2)$ 
of Bernoulli variables with parameter $1-\eta(\de)$. 
Having constructed the subsequence $(V_0,V_1,\dots,V_{n-1})$, the set $V_n$ (or more
precisely the  set of its indicator functions) dominates stochastically
the  $n$th generation of $\bY$. This holds inductively for all $n$, and the claim follows.

\emph{When $\rho\in(0,\oo)$}, we extend the earlier argument (around \eqref{g37}
and later). Rather than presenting all the required details, 
we consider the special case of \eqref{g37};
the general case is similar.
Let $\by\in N$ and $X_t:=\ws_t(\zeta_0)\cap C_\by$.
We develop the previous reference to the first hitting time $F(\zeta_0,\by)$
with a consideration of the limit set $X_\oo=\lim_{t\to\oo}X_t$. 
Since $\ol\zeta_0$ is recurrent and $\zeta_0$ is transient, there exists a deterministic 
$\eta>0$ such that:
\begin{letlist}
\item a.s., $X_\oo$ contains  infinitely many disjoint closed
connected regions $R_1,R_2,\dots$, each with volume exceeding $\frac12 ca$, and
\item every point $\bx\in\bigcup_i R_i$ is such that 
\begin{equation}\label{eq:new123}
|\{t\ge 0: \bx\in \ws_t(\zeta_0)\}|_1 \ge \eta.
\end{equation}
\end{letlist}
Each such region contains a point of $\Pi$ with probability at least 
$1-e^{-\frac12\la ca}$. Each such point is infected by $P_0$ with probability
at least $1-e^{-\rho\eta}$. Pick $N$ such that, in $N$ independent trials 
each with probability
of success $1-e^{-\frac12\la ca}-e^{-\rho\eta}$, there exists at least one success
with probability exceeding $1-\eps$. Finally, pick the 
deterministic time $\tau$ such that there
is probability at least $1-\eps$ that $X_\tau$
contains at least $N$ disjoint closed
connected regions $R_j$, each with volume exceeding $\frac12 ca$, and such that, for every $j$
and every $\bx\in R_j$, inequality \eqref{eq:new123} holds.

Finally, we pick $\alpha$ such that 
$$ 
\PP(T_0>\tau)\ge 1-\eps.
$$
With these choices, the probability that $X_\tau$
contains some particle that is infected from $0$ is at least $(1-\eps)^3$.
The required argument proceeds henceforth as before.

We turn finally to the case of general $\mu$ and $\rho\in(0,\oo)$, 
and we indicate briefly how the method of Section \ref{ssec3.6.3} may be applied in the current context.
First, let $\by\in N$ and $a>3$. It suffices as above  to show that, with probability near $1$,  
$P_0$ infects some particle in $C_\by:=\ol S_\by\times\RR$ where, as usual, $S_\by=3a\by +aS$. 
The following argument is illustrated in Figure \ref{fig:bwz}.

\begin{figure}
\includegraphics[width=0.7\textwidth]{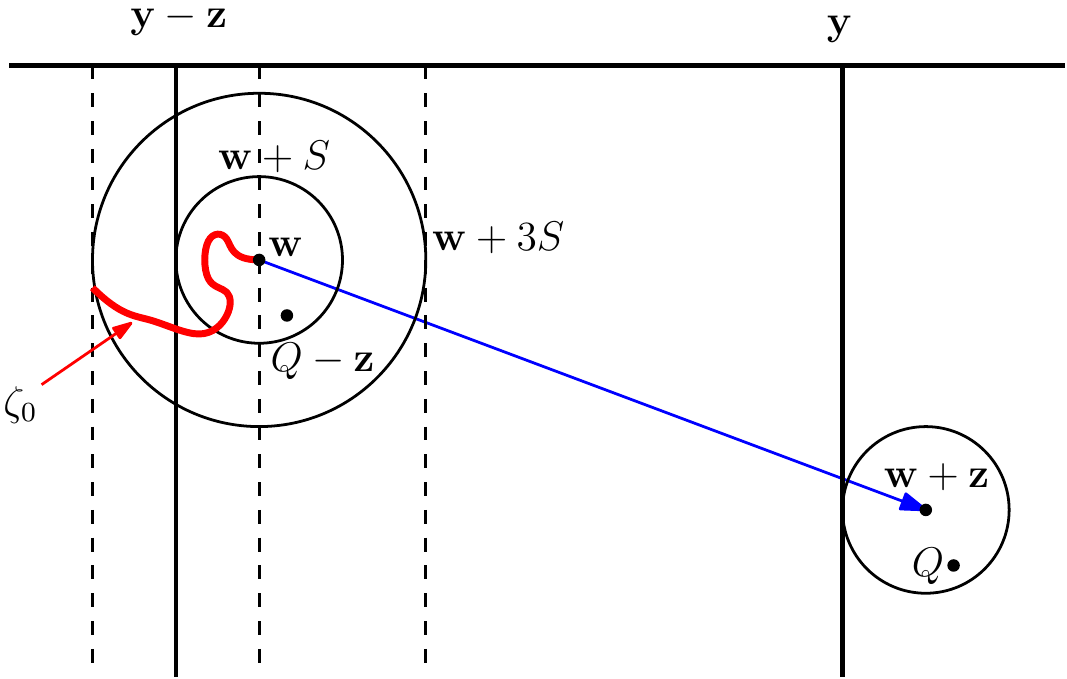}
\caption{An illustration of the proof that, 
with strictly positive probability (and, therefore, by iteration, with probability $1$) $P_0$ infects
some particle in $C_\by$. The arrow denotes the vector $\bz$.}
\label{fig:bwz}
\end{figure}

Pick $\bz\in\RR^3$ and $\eta>0$ such that the $d=3$ version of \eqref{eq:new125} holds, namely,
\begin{equation}\label{eq:new126}
\int_{\bz+S} \mu(\bv)\, d\bv \ge \eta \mu(\bz)>0.
\end{equation}
By recurrence, the projected diffusion $\ol\zeta_0$  visits the disk $\ol S_{\by-\bz}$ infinitely often,
a.s., and therefore $\zeta_0$ visits the tube $T:=\ol S_{\by-\bz}\times\RR$ similarly. 
By the transitivity of $\zeta_0$, its entry points into $T$ are a.s.\ unbounded. Following each such entry
to $T$, at the point $\bw\in \RR^3$ say, there is an exit from the ball $\bw + 3S$. 
Let $H$ be the time of the first such entry and $H'$ the time of 
the subsequent such exit. 

Let $\zeta>0$ denote the volume of the ball $S$, so that%, and let
\begin{equation}\label{eq:new129}
\PP\bigl([\bw+\bz+S]\cap\Pi\ne\es\bigr) \ge 1-e^{-\la\zeta}.
\end{equation}
On the event that $[\bw+\bz+S]\cap\Pi\ne\es$, let $Q$ be the least point in that intersection,
so that $Q\in C_\by$. Conditional
on $Q$, the probability that $P_0$ infects the particle at $Q$ during the time-interval $(H,H')$ is
\begin{equation}\label{eq:new130} 
p := 1- \EE \left[\exp\left( -\rho\int_{H}^{H'} \mu(Q-\zeta(t))\,dt\right)\right].
\end{equation}
By spherical symmetry and \cite[Thm 3.31]{MorP}, conditional on $\bw$,
$$
\EE\int_{H}^{H'} \mu(Q-\zeta(t))\,dt = \int_{\bw+ 3S} \mu(Q-\bu)G(\bw,\bu)\,d\bu,
$$
where $G$ is the appropriate Green's function of \cite[Lem.\ 3.32]{MorP}.
There exists $c>0$ such that $G(\bw,\bu)\ge c$ for $\bu\in \bw  + 2S$. We make
the change of variable $\bu=Q-\bx+\bv$, and note that $Q-\bz+S\subseteq \bw+2S$, to deduce that
\begin{equation}\label{eq:new131}
\EE\int_{H}^{H'} \mu(Q-\zeta(t))\,dt \ge c \int_{\bz +S} \mu(\bv)\,d\bv
\ge  c\eta\mu(\bz)   >0,
\end{equation}
by \eqref{eq:new126}.
By \eqref{eq:new130}, we have $p>0$.

On combining \eqref{eq:new129} and \eqref{eq:new131}, we deduce that there exists $\de>0$ such that
\begin{equation}\label{eq:new132}
\PP\bigl(\text{$\exists Q\in \Pi\cap C_\by$, and $P_0$ infects $Q$ between times $H$ and $H'$}\bigr) \ge\de.
\end{equation}
The proof is completed by using the iterative argument around \eqref{eq:new123}.

\section{The diffusion model}\label{sec:d}

\subsection{A condition for subcriticality}\label{sec:dd-main22}

We consider the diffusion model in the general form
of Sections \ref{sec:3-1-0} and \ref{sec:3-1c},
and we adopt the notation of those sections. Recall the critical 
point $\lac$ of the Boolean continuum percolation on $\RR^d$ in which a
closed unit ball is centred at each point
of a rate-$\la$ Poisson process on $\RR^d$. We shall prove
the existence of a subcritical phase.

Condition \eqref{saus-cond} is now replaced as follows.
Let $\zeta'$ be an independent copy of $\zeta$, and 
define the sausage
\begin{equation}\label{sausage2}
\Si'_t:= \bigcup_{s\in [0,t]} \bigl[\zeta(s) -\zeta'(s)+ S\bigr],\qq  s \ge 0.
\end{equation}
We shall assume
\begin{equation}\label{saus-cond2}
\text{\Cps: for $t \ge 0$, $\EE|\Si'_t|_d \le \ga e^{\si t}$,}
\end{equation}
for some $\ga,\si\in[0,\oo)$, and we make a note about this condition in Remark \ref{rem:4}.

Let $\thd(\la,\rho,\a)$ be the probability that the diffusion process survives.

\begin{theorem}\label{thm:2xx}
Consider the general diffusion model on $\RR^d$ where $d \ge 1$.
\begin{letlist}

\item Let $\rho\in(0,\oo)$ and $\ua(\la,\rho)=\la\rho\,\iNT(\mu)$.
Then 
$\thd(\la,\rho,\alpha)  =0$ if $\alpha>\ua(\la,\rho)$.

\item Let $\rho=\oo$ and $\mu=1_S$. Assume in addition that condition \Cps\ 
of \eqref{saus-cond2} holds. Let $\ua(\la)=\si/(1-\la\ga)$ and $\ula=1/\ga$.
Then $\thd(\la,\oo,\alpha)  
=0$ if $\alpha>\ua(\la)$ and $0<\la<\ula$.
\end{letlist}
\end{theorem}

This theorem extends Theorem \ref{thm:2}.
Its proof is related to that given in Section \ref{sec:3-2} for the delayed 
diffusion model.  
 
\begin{proof}
(a) Let $\la\in(0,\oo)$, and
suppose that $\rho<\oo$. 
We shall enhance the probability space on which the diffusion model is defined.
Let $(T_i: i\in \ZZ_0)$
be random variables with the exponential distribution with parameter $\alpha$; these are independent
of one another and of all other random variables so far. We call $T_i$ the \lq lifetime' of $P_i$,
and it is the length of the period between infection and removal of $P_i$.

For $i\ne j$, we introduce Poisson processes $A_{i,j}$ of points in $[0,\oo)$, 
and we say that $P_i$ \lq contacts' $P_j$ at the times of $A_{i,j}$.
The intensity functions of the $A_{i,j}$ depend as follows on the positions of $P_i$ and $P_j$.
Conditional on $\Pi$ and the diffusions $(\zeta_r:r\in\ZZ_0)$,
let  $(A_{i,j}: i,j\in\ZZ_0,\, i\ne j)$ be independent Poisson processes  on $[0,\oo)$ with respective rate functions 
$$
r_{i,j}(s):=\rho\mu\bigl(X_j+\zeta_j(s)-X_i-\zeta_i(s)\bigr), \qq s \ge 0.
$$
The points of $A_{i,j}$ are denoted $(A_{i,j}^k: k\in\ZZ_0)$.
Let 
$$
\ul A_{i,j}(t):=\inf\{A_{i,j}:  k \in \ZZ_0,\, A_{i,j}^k>t\},\qq t>0,
$$
and let $B_{i,j}(t)$ be the 
event that $\ul A_{i,j}(t)-t<T_i$ and $P_j$ is
susceptible at all times $\ul A_{i,j}(t)-\eps$ for $\eps>0$.
Suppose that $P_i$ becomes infected at time $\tau$.
The first contact by $P_i$ of $P_j$ after time $\tau$ results in an infection
if and only the event $B_{i,j}(\tau)$ occurs (in which case we say that $P_i$
infects $P_j$ \emph{directly}).
Write $\ul A_{i,j}=\ul A_{i,j}(0)$ and $B_{i,j}=B_{i,j}(0)$.

Proposition \ref{prop:1-1} holds with the same proof but with $L_t(x)$ replaced by
\begin{equation}\label{g0-1b}
\wt L_t(x)=
\EE\left(1-\exp\left(-\int_0^t \rho \mu(x+\zeta(s)-\zeta'(s))\,ds\right)\right), 
\end{equation}
where $\zeta'$ is an independent copy of $\zeta$. By the Poisson colouring theorem,
$\wt L_t(x)$ equals the probability that $P_0$ contacts a particle started at $x \in \RR$  
during the time interval $(0,t]$.
With this new $\wt L_t(x)$, the new bound $R=R(\rho)$ now satisfies
 \begin{equation}\label{g7-1zz}
R(\rho) =
\la \int_{\RR^d} \int_0^\oo \wt L_s(x)\alpha e^{-\alpha s} \, ds\, dx \le \frac{\la\rho}{\alpha}\,\iNT(\mu).
\end{equation}
In other words, $R(\rho)$ is the mean number of particles that $P_0$ contacts during its lifetime
(it is \emph{not} the mean total number of contacts by $P_0$, since $P_0$ may contact
any given particle many times).

By an inductive definition as before, we define the $n$th generation $I_n$ of infected particles from $0$. 
We claim that
\begin{equation}\label{eq:441}
\EE|I_n| \le R(\rho)^n, \qq n \ge 1.
\end{equation}
By \eqref{eq:441}, $\EE|I|<\oo$ whenever $R(\rho)<1$,  
and the claim of part (a) follows by \eqref{g7-1zz} as in the proof of
Proposition \ref{prop:2-1}(b,\,c). We turn therefore to the proof of \eqref{eq:441},
which we prove first with $n=1$. 

Recall that each label $i\in\ZZ_0$ corresponds to a point $X_i\in\Pi$, an associated
diffusion $\zeta_i$, and a lifetime $T_i$. The lifetime $T_i$ is the 
residual time to removal of
$P_i$ after its first infection.

We have that
\begin{equation}\label{eq:446}
|I_1| = \sum_{j\in \ZZ_0\sm\{0\}} 1(B_{0,j})
\le \sum_{j\in \ZZ_0\sm\{0\}} 1(\ul A_{0,j}<T_0),
\end{equation}
whence, by the remark after \eqref{g7-1zz},
\begin{equation}\label{eq:442}
\EE|I_1| \le \sum_{j\in \ZZ_0\sm\{0\}} \PP(\ul A_{0,j}<T_0) = R(\rho),
\end{equation}
as claimed.

Suppose next that $n \ge 2$. We introduce some further notation. Let $i_0=0$,
and let 
$\vec\imath=(i_1,i_2,\dots,i_n)$ be an ordered vector of distinct
members of $\ZZ_0\sm\{0\}$; we shall consider $\vec\imath $ as both
a vector and a set.
Define the increasing sequence $\tau(\vec\imath)=(\tau_j: 0\le j \le n)$ of times by
\begin{equation}\label{eq:new152}
\tau_0=0, \q \tau_1=\ul A_{i_0,i_1},\q \tau_2=\ul A_{i_1,i_2}(\tau_1),
\q\dots, \q\tau_{j+1}=\ul A_{i_{j},i_{j+1}}(\tau_{j}).
\end{equation}
By iterating the argument leading to \eqref{eq:446}, we obtain
\begin{equation}\label{eq:445}
|I_n| \le W_n,
\end{equation}
where
\begin{equation}\label{eq:444}
W_n=\sum_{\vec\imath } f(\vec\imath ),
\end{equation}
and
\begin{equation}\label{eq:455}
f(\vec\imath ) = 1(\tau_1<T_{i_0}) 1(\tau_2 -\tau_1< T_{i_1}  ) \cdots 
1(\tau_{n}-\tau_{n-1}< T_{i_{n-1}}).
\end{equation}
Equations \eqref{eq:445}--\eqref{eq:444} 
are implied by the following observation: if $P_{i_n}\in I_n$, then
there exists a sequence $i_0=0,i_1,\dots,i_{n-1}$ such that, for $0 \le j<n$, 
$P_{i_j}$ infects $P_{i_{j+1}}$ directly at the time $\tau_{j+1}$.
See Figure \ref{fig:3}.

\begin{figure}
\includegraphics[width=0.5\textwidth]{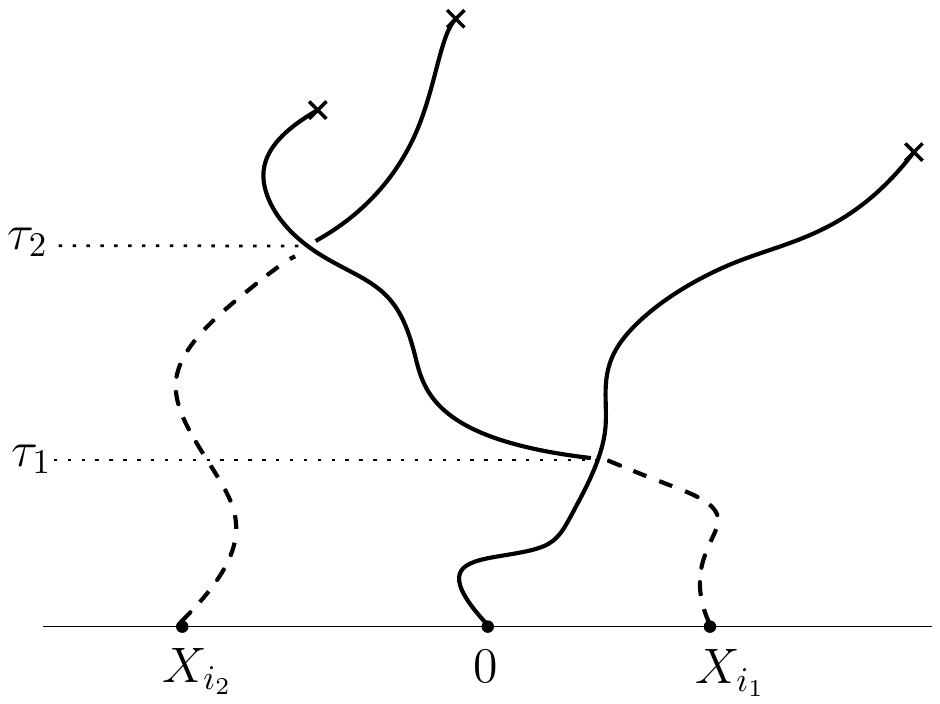}
\caption{The horizontal axis represents one-dimensional space $\RR$, 
and the vertical axis represents time.
This is an illustration of the summand $f(0,i_1,i_2)$ in \eqref{eq:444}
when $d=1$.
In this conceptual view, infections occur where pairs of diffusions intersect, 
and times of removal are marked by crosses.}
\label{fig:3}
\end{figure}

By \eqref{eq:444},
$$
\EE(W_n) \le \sum_{\vec\imath } \EE\bigl[\PP(C_1\cap C_2\cap \cdots 
\cap C_{n}\mid \sG(\vec\imath ))\bigr],
$$
where $C_j=\{\tau_{j}-\tau_{j-1}<T_{i_{j-1}}\}$
and $\sG(\vec\imath )$ is the $\si$-field generated by the random variables
$$
(X_{i_j}, \zeta_{i_j}, T_{i_j}) \text{ for } 0\le j <n-1,
\q X_{i_{n-1}},\ \tau_{n-1},\  (\zeta_{i_{n-1}}(s): s\in [0,\tau_{n-1}]).
$$
Note that $C_1,C_2,\dots,C_{n-1}$ are $\sG(\vec\imath )$-measurable,
so that
$$
\EE(W_n) \le \sum_{\vec\imath } \EE\bigl[1(C_1\cap\cdots\cap C_{n-1})
\PP(C_{n}\mid\sG(\vec\imath ))\bigr].
$$
Therefore,
\begin{equation}\label{eq:448}
\EE(W_n) \le \EE\left[ \sum_{i_1,\dots,i_{n-1}}
1(C_1\cap\cdots\cap C_{n-1})
\sum_{i_n}\PP(C_{n}\mid\sG(\vec\imath ))\right]
\end{equation}
where the summations are over distinct $i_1,\dots,i_n \ne 0$.

It is tempting to argue as follows.
The diffusions $(\zeta_k: k \notin \{i_0,\dots, i_{n-1}\})$ 
are independent of $\sG(\vec\imath )$, and $\tau_{n-1}$ is
$\sG(\vec\imath )$-measurable. 
By the Poisson displacement theorem (see \cite[Sec.\ 5.2]{K93}), 
the positions $\Pi'=(X_k+\zeta_k(\tau_{n-1}): k \notin \{i_0,\dots,i_{n-1}\})$ 
are a subset of a rate-$\la$ Poisson process.
It follows that
\begin{equation}\label{eq:403}
\sum_{i_n} \PP(C_n\mid \sG(\vec\imath )) \le R(\rho).
\end{equation}
By \eqref{eq:444}--\eqref{eq:403},
\begin{equation}\label{eq:701}
\EE(W_n) \le \EE(W_{n-1})R(\rho).
\end{equation}
Inequality \eqref{eq:441} follows by  iteration and \eqref{eq:445}
There is a subtlety in the argument leading to \eqref{eq:403}, 
namely that the distribution of the subset
$(X_k: k \notin \{i_1,\dots,i_{n-1}\})$ of $\Pi$ 
will generally depend on the choice of $i_1,\dots,i_{n-1}$. 
This may be overcome as follows.

We decouple the indices of particles and their starting positions in a classical way
(see \cite[Thm 6.13.11]{GS20}) by giving a more prescriptive recipe for the construction of the Poisson process $\Pi$. 
Let $m$ be a positive integer and let $\La_m=[-m,m]^d\subset \RR^d$; 
later we shall take the limit as $m\to\oo$. Let $M$ have the Poisson distribution with parameter
$\la(2m)^d$. Conditional on $M$, let $X_1,X_2,\dots,X_M$ be independent random variables
with the uniform distribution on $\La_m$. 
Thus, points in $\Pi\cap \La_m$  are indexed $\{0\}\cup J$ where $J=\{1,2,\dots,M\}$,
 with $P_0$ retaining the index $0$.

Let
\begin{equation}\label{eq:449}
W_n(m)=\sum_{\vec\imath \subseteq J} f(\vec\imath ),
\end{equation}
so that $W_n(m)\to W_n$ as $m\to\oo$, and furthermore, 
\begin{equation}\label{eq:new451}
\EE(W_n(m)) \to \EE(W_n)\qq\text{as } m\to\oo,
\end{equation}
by the monotone convergence theorem. 
The sum $W_n(m)$ may be represented in terms of the average of $f(\vec S_n)$ where
$\vec S_n$ is a random ordered $n$-subset of indices in $J$, namely,
\begin{equation}\label{g700}
W_n(m)=\EE\left(\frac{M!}{(M-n)!}f(\vec S_n)\right).
\end{equation}
The term $f(\vec S_n)$ is intepreted as $0$ if $n>M$.
With $\vec S_n=(s_1,s_2,\dots,s_n)$ and $\vec S_{n-1}=(s_1,s_2,\dots, s_{n-1})$,
we have as in \eqref{eq:448} that
\begin{equation}\label{eq:new453}
W_n(m) %= \EE\bigl[ \EE(f(\vec S_n) \mid \sG(\vec S_n), M )\big]
=\EE\left(\frac{M!}{(M-n)!} f(\vec S_{n-1})Z_n\right)
\end{equation}
where 
\begin{align}\label{eq:new452}
Z_n =  1( \tau_{n} -\tau_{n-1}<T_{s_{n-1}}),
\end{align}
and $\tau(\vec S_n) = (\tau_0,\tau_1,\dots,\tau_n)$.

For an ordered $(n-1)$-subset $\vec\imath$ of $J$, let  $\ol R(\vec \imath)$ be the supremum 
over $s\in\La_m$ of the mean number of particles infected by a given initial particle located
at $s$, in the subset of the rate-$\la$ Poisson process obtained from $\Pi$ by 
deleting $(X_i: i\in\{0\}\cup \vec \imath)$. Note that 
\begin{equation}\label{eq:new600}
\ol R(\vec\imath)\le R(\rho).
\end{equation}
By \eqref{eq:new453}--\eqref{eq:new452},
\begin{align*}
\EE(Z_n\mid \sG(\vec S_{n-1}),M)  &= \frac 1{M-n+1} \sum_{s\in J\setminus \vec S_{n-1}} 
\PP\bigl( \tau_{n} -\tau_{n-1}<T_{s_{n-1}} \bigmid \sG(\vec S_n), M\bigr)\\
&\le \frac 1{M-n+1} \ol R(\vec S_{n-1}).
\end{align*}
By \eqref{eq:new453} and \eqref{eq:new600},
\begin{align*}
W_n(m) &\le\EE\left(\frac{M!}{(M-n+1)!}f(\vec S_{n-1}) \ol R(\vec S_{n-1})\right)\\
&=\sum_{\vec\imath\subseteq J} f(\vec \imath) \ol R(\vec\imath) \le W_n(m-1) R(\rho).
\end{align*}
By \eqref{eq:new451}, on letting $m\to\oo$,
we deduce inequality \eqref{eq:701}, and the proof is completed as before.

(b) Let
$\rho=\oo$. We repeat the argument in the proof of part (a) (cf.\ Section \ref{sec:3-3z})
with $R(\oo)$ defined as the mean number of particles $P_j$ for which there exists $t < T_0$ with
$X_j+\zeta_j(t)\in \zeta_0(t)+S$. That is, with $\zeta'$ an independent copy of $\zeta$,
\begin{align}\label{g90}
R(\oo) &=  \int_{\RR^d}\la\,dx\, \PP\bigl(x+\zeta'(t)-\zeta(t)\in S \text{ for some } t<T_0\bigr)\\
&= \int_{\RR^d}\la\, dx\, \int_0^\oo \PP(x \in \Si_s' ) \alpha e^{-\alpha s}\,ds \nonumber\\
&= \la\int_0^t \EE|\Si'_s|_d \,\alpha e^{-\alpha s}\,ds, \nonumber
\end{align}
where $\Si_s'$ is given in \eqref{sausage2}.
As in Theorem \ref{thm:4}(b) adapted to the diffusion model, we have by \Cps\ that $R(\oo)<1$ if 
$\la<\ula:=1/\ga$ and $\alpha>\ua(\lambda):= \si/(1-\la\ga)$.
By the argument of the proof of part (a), 
$\thd(\la,\rho,\alpha)=0$ for $\la\in(0,\ula)$ and $\alpha>\ua(\lambda)$.
\end{proof}

\begin{example}[Bounded motion]\label{ex1b}
Let $\rho=\oo$ and $\mu=1_M$ as above, and suppose 
in addition that each particle is confined within some given
distance $\De<\oo$ of its initial location.
By \eqref{g90},
\begin{equation*}
R(\oo)  \leq \la \bigl| S\bigl(2(\De+\rad(M))\bigr)  \bigr|_d.
\end{equation*}
If the right side is strictly less than $1$, 
then $\thd(\la,\oo,\alpha)=0$ by Proposition \ref{prop:2-1}(b)  
adapted to the current context. 
\end{example}

\begin{remark}[Condition \Cps]\label{rem:4}
Let $M_t=\sup\{ \|\zeta(s)\|_d: s \in [0,t]\}$, the maximum displacement of $\zeta$ up to time $t$,
and let $M'_t$ be given similarly in terms of $\zeta'$.  By Minkowski's inequality,
$$
\EE|\Si_t'|_d \le \EE\bigl([M_t+M'_t+1]^d\bigr)
\le \bigl(2\|M_t\|+1\bigr)^d,
$$
Here, 
$\|\cdot\|$ denotes the $L^d$ norm.
Therefore, \Cps\ holds
for some $\ga$, $\si$ if $\|M_t\| \le \ga' e^{\si'  t}$ for suitable $\ga'$, $\si'$.
\end{remark}

\subsection{The Brownian diffusion model}\label{sec:3-3b}

Suppose that $\rho\in(0,\oo]$,
$\mu=1_S$,
and $\zeta$ is a standard Brownian motion (one may allow it to have
constant non-zero drift, but for simplicity we set the drift to $0$).
Since $(\zeta-\zeta')/\sqrt 2$ 
is a standard Brownian motion, it is easily seen that $\EE|\Si_s'|_d = \EE|W_{2s}|_d$ where
$W$ is the usual radius-$1$ Wiener sausage. Therefore,
$$
R(\oo)=\la \int_0^\oo \EE|W_{2s}|_d\, \alpha e^{-\alpha s}\,ds
= \la \int_0^\oo \EE|W_{s}|_d\, (\alpha/2) e^{-\alpha s/2}\,ds.
$$
Hence, $\ua(\la)=2\oladd(\la)$ where
$\oladd(\la)$ is the corresponding quantity $\oac$ of Example \ref{ex:bmd} for the delayed diffusion model.

\subsection{Survival}\label{rem:6}

We close with some remarks on the missing \lq survival' parts of Theorems \ref{thm:2} 
and \ref{thm:2xx}. An iterative construction similar to that of Section \ref{sec:3} may be explored
for the diffusion model. However, Proposition \ref{lem:3} is not easily extended or adapted
when the particles are \emph{permanently removed} following  infection.  

The situation is different
when either there is no removal (that is, $\alpha=0$, see \cite{BDDHJ}), 
or \lq recuperation' occurs in that particles become susceptible again post-infection. A model of the latter type, but
involving random walks rather than Brownian motions, has been studied by Kesten and Sidoravicius 
in their lengthy and complex work \cite{KS06}. Each of these variants has 
structure not shared with our diffusion model, including
the property that the set of
infectives increases when the set of initially infected particles increases.
Heavy use is made of this property in \cite{KS06}. Unlike the delayed diffusion model 
(see the end of Section \ref{sec:dd-main} and Proposition \ref{prop:4}), 
neither the diffusion model nor its random-walk version has this property, in contradiction of
the claim of Remark 4 of  \cite{KS06}.

\section*{Acknowledgements}
ZL's research was supported by National Science Foundation grant 1608896 and Simons Foundation grant 638143.
GRG thanks Alexander Holroyd and James Norris for useful conversations. The authors are very grateful
to three referees for their detailed and valuable reports, which have led to significant corrections. The work reported here
was influenced in part by the \covid\ pandemic of 2020.

\input{pcv-ejp-rev3.bbl}

%\bibliography{pcv6}
%\bibliographystyle{amsplain}

\end{document}

%% file: pcv-ejp-rev3.bbl
\providecommand{\bysame}{\leavevmode\hbox to3em{\hrulefill}\thinspace}
\providecommand{\MR}{\relax\ifhmode\unskip\space\fi MR }
\providecommand{\MRhref}[2]{%
  \href{http://www.ams.org/mathscinet-getitem?mr=#1}{#2}
}
\providecommand{\href}[2]{#2}